\documentclass[a4paper, 12pt]{amsart}
\usepackage[utf8]{inputenc}
\usepackage[T1]{fontenc}
\usepackage{pdfsync}
\usepackage{hyperref}
\usepackage{bbm}
\usepackage[numbers,square]{natbib}
\usepackage{color}
\usepackage{amsmath}
\usepackage{amssymb}
\usepackage{wasysym}
\usepackage{amsthm}

  \evensidemargin
\oddsidemargin
\usepackage{bbm} 
\newcommand{\ind}{\mathbbm{1}}

\newcommand{\INT}[2]{\int\limits_{#1}^{#2} } 
\renewcommand{\epsilon}{\varepsilon}
\newcommand{\eps}{\epsilon}

\renewcommand{\theta}{\vartheta}
\renewcommand{\phi}{\varphi}
\renewcommand{\leq}{\leqslant}
\renewcommand{\geq}{\geqslant}
\newcommand{\D}{\, \mathrm{d}}
\newcommand{\ton}{\overset{}{\underset{n\to\infty}\longrightarrow}}

\newcommand{\bD}{\mathbb{D}}
\newcommand{\R}{\mathbb{R}}
\newcommand{\CN}{\mathbb{C}} 
\newcommand{\I}{\mathrm{i}}

\renewcommand{\d}{\mathrm{d}}

\newcommand{\e}{\mathrm{e}}

\newcommand{\EXP}[1]{\exp\left(#1\right)}

\DeclareMathOperator{\VAR}{Var}

\DeclareMathOperator{\RE}{Re}

\renewcommand{\l}{\left}
\renewcommand{\r}{\right}

\newcommand{\Var}{\mathop{\mathrm{Var}}\nolimits}

\renewcommand{\Re}{\operatorname{Re}}
\renewcommand{\Im}{\operatorname{Im}}

\newcommand{\todistr}{\overset{d}{\underset{n\to\infty}\longrightarrow}}
\newcommand{\todistrnon}{\overset{d}{\underset{}\longrightarrow}}

\newcommand{\toweaka}{\overset{w}{\underset{}\longrightarrow}}

\newcommand{\dd}{{\rm d}}

\theoremstyle{plain}
\newtheorem{theorem}{Theorem}[section]
\newtheorem{lemma}[theorem]{Lemma}
\newtheorem{corollary}[theorem]{Corollary}
\newtheorem{proposition}[theorem]{Proposition}

\theoremstyle{definition}

\newtheorem{example}[theorem]{Example}

\theoremstyle{remark}

\renewcommand{\P}[1]{\mathbb{P}\left[#1\right]} 

\newcommand{\N}{\mathbb{N}}
\newcommand{\E}{\mathbb{E}}

\newcommand{\EW}[1]{\E\left[#1\right]}

\DeclareMathOperator{\IM}{Im}
\DeclareMathOperator{\COV}{Cov}

\begin{document}

\author{Hendrik Flasche}
\address{Institut f\"ur Mathematische Stochastik,
Westf\"alische Wilhelms-Universit\"at M\"unster,
Orl\'eans--Ring 10,
48149 M\"unster, Germany}
\email{hendrik.flasche@uni-muenster.de}

\author{Zakhar Kabluchko}
\address{Institut f\"ur Mathematische Stochastik,
Westf\"alische Wilhelms-Universit\"at M\"unster,
Orl\'eans--Ring 10,
48149 M\"unster, Germany}
\email{zakhar.kabluchko@uni-muenster.de}

\title[Real Zeros of Random Taylor Series]{Expected Number of Real Zeros of Random Taylor Series}

\keywords{Random polynomials, random Taylor series, real zeroes, weak convergence, random analytic functions, functional limit theorem}

\subjclass[2010]{Primary: 	30C15, 26C10; secondary: 60F99, 60F17, 	60F05, 60G15}

\begin{abstract}
Let $\xi_0,\xi_1,\ldots$ be i.i.d.\ random variables with zero mean and unit variance. Consider a random Taylor series of the form
$
f(z)=\sum_{k=0}^\infty \xi_k c_k z^k,
$
where $c_0,c_1,\ldots$ is a real sequence such that $c_n^2$ is regularly varying with index $\gamma-1$, where $\gamma>0$. We prove that
$$
\mathbb E N[0,1-\epsilon] \sim \frac{\sqrt{\gamma}}{2\pi} |\log \eps| \;\;\; \text{ as } \;\;\;  \varepsilon\downarrow 0,
$$
where  $N[0,r]$ denotes the number of real zeroes of $f$ in the interval $[0,r]$.
\end{abstract}

\maketitle

\section{Introduction and main result}
\subsection{Introduction}
Let $\xi_0,\xi_1,\ldots$ be independent, identically distributed random variables with real values. Consider random polynomials of the following form:
$$
P_n(z) = \sum_{k=0}^n \xi_k z^k, \quad z\in\CN.
$$
\citet{bloch_polya} and~\citet{littlewood_offord1,littlewood_offord2,littlewood_offord3} obtained first estimates on the number of real zeroes of $P_n$.
In the case when the $\xi_k$'s are standard normal, Kac~\cite{kac_explicit} computed explicitly  the expected number of real zeroes of $P_n$ and proved that asymptotically it behaves like $\frac 2 \pi (1+o(1)) \log n$, as $n\to\infty$. The same asymptotics was shown to hold for some other classes of distributions by~\citet{kac_asympt}, \citet{erdoes_offord} and~\citet{stevens}, but it was only in 1971 when Ibragimov and Maslova~\cite{ibragimov_maslova1} proved it when the $\xi_k$'s have arbitrary zero mean distribution from the domain of attraction of the normal law.
The case when the expectation of the $\xi_k$'s is non-zero was considered in~\cite{ibragimov_maslova2},  the asymptotics of the variance and the central limit theorem were obtained in~\cite{maslova_variance} and~\cite{maslova_distribution}, respectively. Under additional assumptions on the distribution of the $\xi_k$'s, \citet{do_nguyen_vu_repulsion}, see also~\cite{nguyen_nguyen_vu}, proved that the expected number of real roots is $\frac 2 \pi \log n + C + o(1)$. Assuming only that the $\xi_k$'s are non-degenerate and exchangeable, Ken S\"oze proved the upper bound $C\log n$ on the expected number of real roots, thus confirming a conjecture of L.\ Shepp. Regarding the complex zeroes, Ibragimov and Zaporozhets~\cite{ibragimov_zaporozhets} proved that  their empirical measure weakly converges to the uniform distribution on the unit circle a.s.\ if and only if $\E \log_+ |\xi_1|$ is finite. The expected number of real roots of random trigonometric polynomials whose coefficients are i.i.d.\ random variables with finite second moment was computed asymptotically  by \citet{flasche}.
Recently, new methods coming from random matrix theory were introduced into the theory of random polynomials by \citet{tao_vu} and developed further by~\citet{do_nguyen_vu_arbitrary}.

In the present paper, we shall be interested in random Taylor series of the form $\sum_{k=0}^\infty \xi_k z^k$, or, more generally,  $\sum_{k=0}^\infty \xi_k c_k z^k$ under a regular variation assumption on the sequence of weights $c_k$ which ensures that the convergence radius of the series is $1$, with probability $1$.  The expected number of real zeroes of such Taylor series is infinite.
Our aim is to describe the speed of clustering of real zeroes of $f$ near the point $1$.

\subsection{Main result}
Let $(\xi_k)_{k\in\N_0}$ be a sequence of independent identically distributed real-valued random variables with
\begin{equation*}
	\E [\xi_0]=0 \quad \text{and} \quad \E[\xi_0^2]=1.
\end{equation*}
Throughout, we assume that the random variable $\xi_0$ is non-degenerate, that is $\P{\xi_0=0}<1$.  Let $(c_k)_{k\in\N_0}$ be a deterministic sequence of real numbers such that
\begin{equation}\label{eq:c_k_reg_var}
c_n^2= \frac{n^{\gamma-1}}{\Gamma(\gamma)}L(n)
\end{equation}
for some $\gamma>0$ and some function $L$ that varies slowly\footnote{A function $L(t)$, defined for $t>0$, is called slowly varying if $L(t)>0$ for sufficiently large $t$ and  $\lim_{t\to +\infty} L(\lambda t)/L(t) = 1$ for all $\lambda>0$. We may and shall assume that $c_0=0$, which does not restrict generality.} at $+\infty$. Here, $\Gamma$ denotes the Gamma function and the term $\Gamma(\gamma)$ is included for convenience.
Our main object of interest is the random Taylor series $f$ given by
\begin{equation*}
	f(z)=\sum_{k=0}^\infty \xi_k c_k z^k.
\end{equation*}
Under our assumptions on the $c_k$'s, $f(z)$ converges on the open unit disk $\bD= \{z\in\CN\colon |z|<1\}$ and
defines an analytic function there, with probability $1$. The number of zeroes of $f$ in any disk of the form $\bD_r = \{z\in \CN\colon |z|<r\}$, where $r<1$, is finite, however the zeroes cluster near the boundary of $\bD$. Let $N[a,b]$, respectively $N[a,b)$, denote the number of real zeroes of $f$ in the interval $[a,b]\subseteq (-1,1)$, respectively $[a,b)\subset (-1,1)$.
The following theorem is our main result.
\begin{theorem}\label{thm:mainresult}
Under the above assumptions, we have
\begin{equation*}
\lim_{r\uparrow 1}\frac{\E N[0,r]}{-\log(1-r)}=\frac{\sqrt{\gamma}}{2\pi}.
\end{equation*}
\end{theorem}
\begin{example}
In the case when $c_k = 1$ for all $k\in\N_0$ we have $\gamma=1$ and Theorem~\ref{thm:mainresult} yields the following asymptotics for the number of real zeroes of the Taylor series $f(z) = \sum_{k=0}^\infty \xi_k z^k$ in the interval $[0,r]$:
$$
\E N[0,r] \sim \frac{1}{2\pi} \log \frac{1}{1-r},\quad  \text{ as } r\uparrow 1.
$$
Here, $a(r)\sim b(r)$ as $r\to r_0$ means that $\lim_{r\to r_0} a(r)/b(r) = 1$.
For the number of critical points, i.e.\ the zeroes of the derivative $f'(z) = \sum_{k=1}^\infty  k \xi_k z^{k-1}$, Theorem~\ref{thm:mainresult}  with $\gamma=2$ yields
$$
\E N_{\text{crit}}[0,r] \sim \frac{1}{\sqrt 2 \,\pi} \log \frac{1}{1-r},\quad  \text{ as } r\uparrow 1.
$$
\end{example}

It seems that only very little is known about real zeroes of random Taylor series.  One  exception is the paper of~\citet[Section~2.5]{do_nguyen_vu_arbitrary} who proved a local universality result for real and complex zeroes.  Our approach is a development of the method of \citet{ibragimov_maslova1} and is independent of the method of~\cite{do_nguyen_vu_arbitrary}. One of the new features compared both to~\cite{ibragimov_maslova1} and~\cite{do_nguyen_vu_arbitrary}  is the use of functional limit theorems. The scope of our method is not restricted to random Taylor series. For example, random trigonometric polynomials were treated in~\cite{flasche}, \cite{iksanov2016}, \cite{angst_etal_bivariat}.
One of the conditions required by~\citet{do_nguyen_vu_arbitrary} was the finiteness of the $(2+\delta)$-th moment of $\xi_0$. We require only the finiteness of the second moment, but even this requirement is not critical.
In fact, the second (and most difficult) part of our proof applies with minimal modifications to the case when $\xi_0$ belongs to an $\alpha$-stable domain of attraction. The first part of the proof (the functional limit theorem) also can be adapted to the $\alpha$-stable case (leading to a different asymptotics for the number of real roots), but we refrain from doing it here.  As in~\cite{ibragimov_maslova1}, Theorem~\ref{thm:mainresult} continues to hold without changes under  the assumption that $\xi_0$ is in the domain of attraction of the normal law. We refrain from giving the proof in this level of generality because it leads to more complicated notation without requiring new ideas.

\vspace*{2mm}
\noindent
\emph{Notation.}   In the following, $C>0$ (respectively, $c>0$) denotes a sufficiently large  (respectively, small) constant that does not depend on $n$ and may change from line to line. Most statements hold for sufficiently large $n\geq n_0$ only, where the number $n_0$ also can change from line to line.   The floor and the ceiling functions of $x$ are denoted by $\lfloor x \rfloor$ and $\lceil x \rceil$, respectively.

\section{Method of proof of Theorem \ref{thm:mainresult}}
\subsection{The main lemmas}
The main body of the paper will be devoted to the proof of the following crucial lemmas that easily imply Theorem \ref{thm:mainresult}:
\begin{lemma}\label{lem:num_zeroes_conv_distr}
Fix some $0<q<1$. As $n\to\infty$, the random variable $N\l[1-q^n,1-q^{n+1}\r)$ converges in distribution to certain random variable with values in $\{0,1,\ldots\}$ and expectation
$$
-\frac{\sqrt{\gamma}\log q}{2\pi}.
$$
\end{lemma}

\begin{lemma}\label{lem:num_zeroes_exp_bounded}
For every $q\in(e^{-1/32},1)$ and all $1 < \kappa <2$ there is $n_0\in\N$ such that
$$
\sup_{n\geq n_0} \E N^\kappa\l[1-q^n,1-q^{n+1}\r] < +\infty.
$$
\end{lemma}

\begin{proof}[Proof of Theorem \ref{thm:mainresult} given Lemmas~\ref{lem:num_zeroes_conv_distr} and~\ref{lem:num_zeroes_exp_bounded}]
Fix some $q\in(e^{-1/32},1)$. Taken together, both lemmas imply by uniform integrability that
\begin{equation}\label{eq:konvergenzerwartungswerteteilintervall}
	\lim_{n\to\infty} \E N\l[1-q^n,1-q^{n+1}\r)=-\frac{\sqrt{\gamma}\log q}{2\pi}.
\end{equation}
The interval $[0,1)$ can be covered by disjoint intervals of the form $[1-q^n,1-q^{n+1})$, where $n=0,1,\ldots$. Hence, we have the following estimates:
\begin{align}\label{eq:est_obenundunten}
\begin{split}
\frac{\E N[0,r]}{-\log(1-r)}
&\leq\frac{1}{-\log q} \cdot  \frac 1 {\log_q(1-r)} \sum_{k=0}^{\lfloor\log_q(1-r) \rfloor} \E N[1-q^k,1-q^{k+1}), \\
\frac{\E N[0,r]}{-\log(1-r)}
&\geq\frac{1}{-\log q} \cdot \frac 1 {\log_q(1-r)} \sum_{k=0}^{\lfloor\log_q(1-r) \rfloor -1} \E N[1-q^k,1-q^{k+1}),
\end{split}
\end{align}
where $\log_q$ denotes the logarithm with base $q$.
Since the Ces\`{a}ro limit of a convergent sequence coincides with its usual limit, \eqref{eq:konvergenzerwartungswerteteilintervall} and~\eqref{eq:est_obenundunten} imply
\begin{align}\label{eq:cesaro_convergence}
\begin{split}
\lim_{r\uparrow 1} \frac 1 {\log_q(1-r)} \sum_{k=0}^{\lfloor \log_q(1-r) \rfloor} \E N[1-q^k,1-q^{k+1})
=
-\frac{\sqrt{\gamma}\log q}{2\pi}, \\
\lim_{r\uparrow 1}\frac 1 {\log_q(1-r)} \sum_{k=0}^{\lfloor\log_q(1-r) \rfloor-1} \E N[1-q^k,1-q^{k+1})
=
-\frac{\sqrt{\gamma}\log q}{2\pi}.
\end{split}
\end{align}
Taken together, \eqref{eq:est_obenundunten} and \eqref{eq:cesaro_convergence}  yield the statement of Theorem \ref{thm:mainresult}.
\end{proof}

\subsection{Method of proof of Lemmas~\ref{lem:num_zeroes_conv_distr} and~\ref{lem:num_zeroes_exp_bounded}}
To prove Lemma~\ref{lem:num_zeroes_conv_distr}, we consider the sequence of stochastic processes
\begin{equation}
X_{q^n}(z):=\frac{f(1-q^n z)}{\sqrt{v(1-q^n z)}}, \quad n\in\N,
\end{equation}
where $v(z)$ denotes the variance of $f(z)$:
\begin{equation*}
v(z):=\EW{f^2(z)}=\sum_{k=0}^\infty c_k^2 z^{2k}.
\end{equation*}
For the time being, $X_{q^n}(z)$ is defined for $z\in [q,1]$, but later we shall continue it analytically to a larger domain.
Clearly, the number of zeroes of $X_{q^n}$ in the interval $[q,1]$ is the same as the number of zeroes of $f$ on $[1-q^n,1-q^{n+1}]$.
In Section \ref{sec:convergence_to_gaussian_process} it will be shown that the process $X_{q^n}$ converges, as $n\to\infty$,  to certain Gaussian process weakly on a suitable space of analytic functions. The expected number of real zeroes of the latter process in $[q,1]$ can be calculated explicitly using the Rice formula. Given the weak convergence of random analytic functions, we can conclude weak convergence of their number of zeroes using the continuous mapping theorem following the method of~\cite{iksanov2016}.

The proof of Lemma~\ref{lem:num_zeroes_exp_bounded} is much more complicated. Essentially, we have to show that the process $X_{q^n}$ cannot have ``too many'' zeroes in $[q,1]$, which is very closely related to showing that the random variables $|X_{q^n}(s)|$, $s\in [q,1]$, cannot be ``too small''; see Lemma~\ref{lemma:abschaetzungdmeins}, below.  Questions of this type are known to be rather difficult in the literature on random matrices and random polynomials. If $\xi_0$ takes the values $\pm 1$ with probability $1/2$ and $c_k=1$, the distribution of $X_{q^n}(s)$ is known as Bernoulli convolution, and the question whether it is singular or absolutely continuous w.r.t.\ the Lebesgue measure is highly non-trivial; see~\cite{solomyak} for a review.
To prove Lemma~\ref{lem:num_zeroes_exp_bounded}, we develop further the ideas from the paper of \citet{ibragimov_maslova1} who considered random polynomials of the form $P_n(z) = \sum_{k=0}^n \xi_k z^k$. Both the presence of weights $c_k$ and the fact that the number of terms in $f(z)$ is infinite  lead to considerable technical problems. At this point, let us mention just one, by far not the most severe, difficulty. In the case of polynomials, the number of zeroes is trivially bounded above by $n$. In the case case of Taylor series, even this trivial bound is not available, and we have to work hard to prove Lemma~\ref{lemma:Nngreaterlambda}, below,  which gives a bound on the truncated expectation of the number of zeroes of $f$ in some interval.

\section{Local convergence to the Gaussian process}\label{sec:convergence_to_gaussian_process}
\subsection{Variance and covariance}
Note that for all $z,w\in \bD$ we have $\E [f(z)] = 0$ and
\begin{equation}\label{eq:covar_f}
\E [f(z) \overline{f(w)}] = \sum_{k=0}^\infty c_k^2 (z\bar w)^{k} = v(\sqrt{z\bar w})
,\quad
\E [f(z) f(w)] = \sum_{k=0}^\infty c_k^2 (z w)^{k} = v(\sqrt{z w}),
\end{equation}
where we defined
\begin{equation}\label{eq:def_v}
v(z):= \sum_{k=0}^\infty c_k^2 z^{2k}, \quad z\in \CN,\;\; |z|<1.
\end{equation}
In particular,
$$
\Var [f(z)] = \E [|f(z)|^2] =  \sum_{k=0}^\infty c_k^2 |z|^{2k} = v(|z|).
$$

The next Abelian theorem is well known. Usually, it is stated for real $z$, see, e.g., \cite[Theorem 5 on p.~423]{feller1966}. The complex version can be found in~\cite{karamata}, \cite{baumann}.
\begin{theorem}\label{theo:abel}
Under condition~\eqref{eq:c_k_reg_var}, we have
$$
v(1 - a z) \sim (2az)^{-\gamma} L\l(\frac{1}{a}\r) \quad \text{ as } a\downarrow 0, \;\; a\in\R,
$$
uniformly as long as $z$ stays in any compact subset of the open right half-plane $\{\Re z >0\}$.
\end{theorem}

\subsection{Functional limit theorem}
Fix some $R>1$ and consider the rectangle
$$
Q_R:= \{z\in\CN\colon \Re z\in [R^{-1},R], |\Im z| \leq R\} \subset \{\Re z >0\}.
$$  Let $\mathcal{H}(Q_R)$ be the Banach space of complex-valued functions which are continuous on $Q_R$ and analytic in the interior of $Q_R$. We endow $\mathcal{H}(Q_R)$ with the usual supremum norm.  Let
$\mathcal H_{\R}(Q_R)$ be the closed subset of $\mathcal H(Q_R)$ consisting of functions which take real
values on $\R\cap Q_R$.
For sufficiently small $a>0$ we define the following random analytic function on $Q_R$:
\begin{equation}\label{eq:defXn}
X_a(z):=\frac{f(1- a z)}{\sqrt{v(1- a z)}}, \quad z\in Q_R.
\end{equation}
Note that $X_a$ is well-defined because for sufficiently small $a>0$, $1- a Q_R$ is a subset of
the unit disc and the function $z\mapsto v(1-az)$ has no zeroes in $Q_R$ by Theorem~\ref{theo:abel}, so that we can take the
principal branch of the square root. A major step in proving Lemma~\ref{lem:num_zeroes_conv_distr} is
the following functional limit theorem.
\begin{theorem}\label{theo:FCLT}
Under the above assumptions, weakly on the space $\mathcal H_{\R}(Q_R)$, we have
\begin{equation}\label{eq:theo:FCLT}
(X_a(z))_{z\in Q_R} \toweaka  (Z_\gamma(z))_{z\in Q_R} \text{ as } a\downarrow 0,
\end{equation}
where $Z_\gamma$ is a random analytic function defined on the right half-plane $\Re z>0$ by
$$
Z_\gamma(z) = \frac{(2z)^{\gamma/2}}{\sqrt{\Gamma(\gamma)}} \int_0^\infty \e^{-zw} w^{(\gamma-1)/2} \dd B(w),
$$
with $(B(w))_{w\geq 0}$ being a standard real-valued Brownian motion.
\end{theorem}
The fact that $Z_\gamma(z)$ is indeed a random analytic function on the right half-plane can
be easily seen by partial integration.  It follows from the above that for all $s,t\in  \CN$, $\Re
s>0, \Re t>0$, we have
\begin{equation*}
\E [Z_{\gamma}(t) \overline{Z_{\gamma}(s)}] = \frac{2^{\gamma}(t\bar s)^{\gamma/2}}{(t+\bar s)^{\gamma}},
\quad
\E [Z_{\gamma}(t) Z_{\gamma}(s)] = \frac{2^{\gamma}(ts)^{\gamma/2}}{(t+s)^{\gamma}}.
\end{equation*}
The restriction of $Z_\gamma$ to $(0,\infty)$ is a real-valued, centered Gaussian process.
We can transform $Z_\gamma$ to a stationary process by considering
$$
Y(s) := Z(\e^s), \quad s\in\CN,  \;\;|\Im s|< \pi/2.
$$
Then, $(Y(s))_{s\in\R}$ is a zero-mean stationary real-valued Gaussian process with covariance function
\begin{equation*}
\COV\l[Y(u),Y(v)\r]=\l(\cosh\l(\frac{v-u}{2}\r)\r)^{-\gamma}, \quad u,v\in\R.
\end{equation*}
The above discussion shows that the process $(Y(s))_{s \in\R}$ admits an analytic continuation to the strip $\{z\in\ \CN\colon |\Im z| < \pi/2\}$. For $\gamma=1$, the process $Y$ appeared in the work of~\citet{dembo_etal}.

\begin{corollary}
We have the following one-dimensional CLT:
\begin{equation}\label{eq:CLT_one_dim}
\frac{f(1-a)}{(2a)^{-\gamma/2} \sqrt{L(1/a)}} \todistrnon N(0,1) \text{ as } a\downarrow 0.
\end{equation}
\end{corollary}
\begin{proof}
Observe that $\varphi\mapsto \varphi(1)$ defines a continuous mapping from $\mathcal H_\R(Q_R)$ to $\R$. The continuous mapping theorem applied to~\eqref{eq:theo:FCLT}, together with the asymptotics $v(1-a) \sim (2a)^{-\gamma} L(1/a)$, $a\downarrow 0$,  that follows from Theorem~\ref{theo:abel}, yields~\eqref{eq:CLT_one_dim}.
\end{proof}
In the special case when $c_k=1$, the above CLT~\eqref{eq:CLT_one_dim}, along with a law of the iterated logarithm, was obtained by~\citet{bovier_picco_pre,bovier_picco}.

\subsection{Proof of Theorem~\ref{theo:FCLT}}
Fix some positive sequence  $(a_n)_{n\in\N}$  such that $\lim_{n\to\infty} a_n = 0$. Our aim is to show that the process  $(X_{a_n}(z))_{z\in Q_R}$ converges to $(Z_\gamma(z))_{z\in Q_R}$ weakly on $\mathcal H_\R(Q_R)$, as $n\to\infty$.

\vspace*{2mm}
\noindent
\emph{Convergence of the finite-dimensional distributions.}
We need to show that for every $d\in\N$ and every $z_1,\ldots, z_d\in Q_R$,
\begin{equation}\label{eq:conv_fidi}
\begin{pmatrix}
\RE X_a(z_1)\\
\IM X_a(z_1) \\
\vdots \\
\RE X_a(z_d) \\
\IM X_a(z_d)
\end{pmatrix}
\stackrel{\d}{\longrightarrow }
\begin{pmatrix}
\RE Z_\gamma(z_1)\\
\IM Z_\gamma(z_1)\\
\vdots \\
\RE Z_\gamma(z_d) \\
\IM Z_\gamma(z_d)
\end{pmatrix}
\quad \textrm{as $a\downarrow 0$}.
\end{equation}
To this end, we shall verify the conditions of the $2d$-dimensional Lindeberg central limit theorem; see Proposition~\ref{lemma:multivariate_lindeberg_CLT}, below.
More precisely, we  consider the following array of random vectors:
\begin{equation*}
V_{n,k}:=\begin{pmatrix}
\xi_kc_k \RE \frac{(1-a_n z_1)^k}{\sqrt{v(1-a_n z_1)}} \\
\xi_kc_k \IM \frac{(1-a_n z_1)^k}{\sqrt{v(1-a_n z_1)}} \\
\vdots \\
\xi_kc_k\RE \frac{(1-a_n z_d)^k}{\sqrt{v(1-a_n z_d)}} \\
\xi_kc_k\IM \frac{(1-a_n z_d)^k}{\sqrt{v(1-a_n z_d)}}
\end{pmatrix}\in \R^{2d}, \quad
n\in\N, k\in\N_0.
\end{equation*}
Observe that
\begin{equation}\label{eq:sum_V_n_k}
 \sum_{k=0}^\infty V_{n,k}= \begin{pmatrix}
\RE X_{a_n}(z_1)\\
\IM X_{a_n}(z_1) \\
\vdots \\
\RE X_{a_n}(z_d) \\
\IM X_{a_n}(z_d)
\end{pmatrix}.
\end{equation}
Note that the random vector on the right-hand side of~\eqref{eq:sum_V_n_k} has $2d$-dimensional, centered Gaussian distribution.  To prove the convergence of the covariances stated in condition~(a) of Proposition~\ref{lemma:multivariate_lindeberg_CLT},  it suffices to verify that
\begin{equation}
\begin{split}\label{eq:conv_expect}
\lim_{n\to\infty}\E [X_{a_n}(z_i) X_{a_n}(z_j)]&=\E[Z_\gamma(z_i) Z_\gamma(z_j)] = \frac{2^\gamma (z_iz_j)^{\gamma/2}}{(z_i+z_j)^{\gamma}}, \\
\lim_{n\to\infty}\E [X_{a_n}(z_i) \overline{X_{a_n}(z_j)}]&=\E[Z_\gamma(z_i)\overline{Z_\gamma(z_j)}] =\frac{2^\gamma (z_i\bar z_j)^{\gamma/2}}{(z_i+\bar z_j)^{\gamma}}
\end{split}
\end{equation}
because the covariance matrices of the $2d$-dimensional random vectors on the right-hand side of~\eqref{eq:sum_V_n_k} can be expressed as linear combinations of the above covariances. The expectation in the first line of~\eqref{eq:conv_expect} is given by
\begin{equation*}
\EW{X_{a_n}(z_i)X_{a_n}(z_j)}=
\frac{v\l(\sqrt{(1-a_n z_i)(1-a_n z_j)}\r)}{\sqrt{v(1-a_nz_i)v(1-a_nz_j)}}  \ton \frac{2^\gamma (z_iz_j)^{\gamma/2}}{(z_i+z_j)^{\gamma}},
\end{equation*}
where we used~\eqref{eq:defXn}, \eqref{eq:covar_f} and Theorem~\ref{theo:abel}.
The second line of~\eqref{eq:conv_expect} follows from the identity  $\overline{X_{a_n}(z)}=X_{a_n}(\overline{z})$.

It remains to verify the Lindeberg condition of Proposition~\ref{lemma:multivariate_lindeberg_CLT}, i.e.\ to prove that  for every $\eps>0$,
\begin{align*}
L_n(\eps):=\max_{i=1,\dots, 2d}\sum_{k=0}^\infty \EW{V_{n,k}^2(i)\ind_{|V_{n,k}(i)|\geq \eps}} \ton 0,
\end{align*}
where $V_{n,k}(i)$ is the $i$-th coordinate of $V_{n,k}$. Since $|\Re z|\leq |z|$ and $|\Im z|\leq |z|$, it suffices to show that
for every $z\in Q_R$,
\begin{equation}\label{eq:hinreichendfuerlindeberg}
\lim_{n\to\infty}\sum_{k=0}^\infty \frac{c_k^2|1-a_nz|^{2k}}{|v(1-a_nz)|} \EW{\xi^2_k \ind_{\{\xi_k^2\geq m_{n,k}\}}}=0,
\end{equation}
where
\begin{equation*}
m_{n,k}=\frac{\eps^2|v(1-a_n z)|}{|1-a_nz|^{2k}c_k^2}.
\end{equation*}
We shall prove that for every fixed $y>0$,
\begin{equation}\label{eq:Lindeberg1}
\lim_{n\to\infty}
\sum_{k=0}^{\lfloor y/a_n\rfloor} \frac{c_k^2|1-a_n z|^{2k}}{|v(1-a_nz)|} \EW{\xi_k^2\ind_{\{\xi_k^2\geq m_{n,k}\}}}
=0,
\end{equation}
and
\begin{equation}\label{eq:Lindeberg2}
\lim_{y\to\infty} \limsup_{n\to\infty} \sum_{k=\lceil y/a_n\rceil}^\infty \frac{c_k^2|1-a_nz|^{2k}}{|v(1-a_n z)|}\EW{\xi_k^2 \ind_{\xi_k^2 \geq m_{n,k}}}=0.
\end{equation}
\begin{lemma}\label{lemma:v_with_abs_val}
For every $R>1$ there exists $c>0$ such that
$
|v(1-a_n z)| \geq c v(|1-a_n z|)
$
for all $z\in Q_R$ and all sufficiently large $n\in\N$.
\end{lemma}
\begin{proof}
Let us first show that there exists $c>0$ such that for all sufficiently large $n$ and all $z\in Q_R$,
\begin{equation}\label{eq:lem3.4}
c|z| \leq \frac{1 - |1-a_n z|}{a_n}\leq |z|.
\end{equation}
The upper estimate follows from the triangle inequality. To prove the lower estimate, observe that since $a_n\to 0$,
\begin{equation}\label{eq:lemma:v_with_abs_val}
|1-a_n z| = \sqrt{1- 2 a_n \Re z + a_n^2 |z|^2} = 1- a_n \Re z  + o(a_n) \leq 1 - \frac {a_n \Re z}2  =  1 -  \frac{\Re z}{2|z|}a_n |z|   ,
\end{equation}
for all $z\in Q_R$ and all sufficiently large $n$. This implies the lower estimate in~\eqref{eq:lem3.4} because $(\Re z) /(2|z|)$ is bounded below for $z\in Q_R$.
According to Theorem~\ref{theo:abel} we have
\begin{align}
|v(1-a_n z)| &\sim |2a_n z|^{-\gamma} L(1/a_n), \label{eq:asymptv1}\\
v(|1-a_n z|) &= v\left(1-a_n \frac {1- |1-a_n z|}{a_n}\right)
\sim 2^{-\gamma} (1- |1-a_n z|)^{-\gamma} L(1/a_n) \label{eq:asymptv2}
\end{align}
as $n\to\infty$.
In fact, by Theorem~\ref{theo:abel}, both~\eqref{eq:asymptv1} and~\eqref{eq:asymptv2} hold uniformly over $z\in Q_R$, where in the latter case we have to employ~\eqref{eq:lemma:v_with_abs_val}.  Taking the quotient of~\eqref{eq:asymptv1} and~\eqref{eq:asymptv2}, recalling that $\gamma>0$ and using~\eqref{eq:lem3.4} we arrive at the required estimate $|v(1-a_n z)| \geq c v(|1-a_n z|)$.
\end{proof}

\vspace*{2mm}
\noindent
\emph{Proof of~\eqref{eq:Lindeberg1}.}
Let $\tilde m_n := \min_{k=0,\dots,\lfloor y/a_n\rfloor } m_{n,k}$. Then, by Lemma~\ref{lemma:v_with_abs_val} and~\eqref{eq:def_v},
\begin{align*}
\sum_{k=0}^{\lfloor y/a_n\rfloor} \frac{c_k^2|1-a_n z|^{2k}}{|v(1-a_nz)|} \EW{\xi_k^2\ind_{\{\xi_k^2\geq m_{n,k}\}}}
&\leq C\sum_{k=0}^{\lfloor y/a_n\rfloor }\frac{c_k^2|1-a_n z|^{2k}}{v(|1-a_nz|)}\EW{\xi_k^2\ind_{\{\xi_k^2\geq m_{n,k}\}}} \\
&\leq C\EW{\xi_0^2\ind_{\{\xi_0^2\geq \tilde{m}_{n}\}}}\sum_{k=0}^{\lfloor y/a_n\rfloor }\frac{c_k^2|1-a_n z|^{2k}}{v(|1-a_nz|)} \\
&\leq C\EW{\xi_0^2\ind_{\{\xi_0^2\geq \tilde{m}_{n}\}}}.
\end{align*}
Since $\E [\xi_0^2] = 1$, it remains to check that $\lim_{n\to\infty} \tilde m_n = +\infty$. In the following let $n$ be sufficiently large, so that, for example,  $|1-a_n z| < 1$. By Lemma~\ref{lemma:v_with_abs_val},
\begin{align*}
\frac 1 {\tilde m_n}
=
\max_{k=0,\dots, \lfloor y/a_n\rfloor } \frac{|1-a_nz|^{2k}c_k^2}{\eps^2 |v(1-a_nz)|}
\leq
C  \frac{\max_{k=0,\dots, \lfloor y/a_n\rfloor } c_k^2}{v(|1-a_nz|)}.
\end{align*}
The definition of $v$, see~\eqref{eq:def_v}, implies
$$
v(|1-a_nz|) \geq \sum_{j=0}^{\lfloor y/a_n\rfloor}|1-a_nz|^{2j} c_j^2 \geq c \sum_{j=0}^{\lfloor y/a_n\rfloor} c_j^2,
$$
because for all $j=0,\ldots, \lfloor y/a_n\rfloor$ we have the estimate $|1-a_nz|^{2j} \geq |1-a_nz|^{2y/a_n} \to \e^{-2y\Re z}$ as $n\to\infty$, hence $\min_{j=0,\ldots,\lfloor y/a_n\rfloor} |1-a_nz|^{2j} > c$ for some $c>0$.
In view of the above estimates, the claim $1/\tilde m_n \to 0$ as $n\to \infty$ becomes a consequence of the following
\begin{lemma}\label{lemma:regvarmaximum}
Let the sequence $(c_k)_{k\in\N_0}$ be as in~\eqref{eq:c_k_reg_var}.
Then,
\begin{equation*}
\lim_{n\to\infty}\frac{\max_{k=0,\dots,n}c_k^2}{c_0^2+\dots+c_n^2}=0.
\end{equation*}
\end{lemma}
\begin{proof}
Fix some  $0<\eps < \min \{1/2, \gamma\}$. It follows from~\eqref{eq:c_k_reg_var} that $c_k^2 \leq k^{\gamma-1+\eps}$ for sufficiently large $k\in\N$. Hence,
\begin{equation}\label{eq:est1}
\max_{k=0,\dots,n} c_k^2 \leq \max\{C,  n^{\gamma-1+\eps}\}.
\end{equation}
On the other hand, \eqref{eq:c_k_reg_var} and Karamata's theorem~\cite[Proposition~1.5.8 on p.~26]{bingham_book} imply
\begin{equation}\label{eq:est2}
c_0^2+\dots+c_n^2 \sim \frac{n^\gamma}{\Gamma(\gamma+1)} L(n) >  n^{\gamma -\eps}.
\end{equation}
for sufficiently large $n$.  Taking the quotient of~\eqref{eq:est1} and~\eqref{eq:est2} proves the lemma because $\eps<1/2$ and $\eps<\gamma$.
\end{proof}

\vspace*{2mm}
\noindent
\emph{Proof of~\eqref{eq:Lindeberg2}.}
The claim is a consequence of the following
\begin{lemma}\label{lemma:asymptotik_Fnk_ystat}
For every fixed $y>0$ and every $z\in \mathbb C$ with $\Re z >0$ we have
\begin{equation*}
\lim_{n\to\infty} \sum_{k=\lceil y/a_n \rceil}^\infty \frac{c^2_k|1-a_nz|^{2k}}{|v(1-a_nz)|} =
\frac{(2|z|)^{\gamma}}{\Gamma(\gamma)}\INT{y}{\infty} x^{\gamma-1} \e^{-2(\Re z)x}\D x.
\end{equation*}
\end{lemma}
\begin{proof} 
Keeping the use of the dominated convergence theorem in mind, we write the sum as the integral
\begin{equation*}
\sum_{k=\lceil y/a_n\rceil}^\infty \frac{c^2_k |1-a_nz|^{2k}}{|v(1-a_nz)|}
=\INT{0}{\infty} g_n(x)\D x,
\end{equation*}
where
\begin{equation*}
g_n(x):=\frac{|1-a_nz|^{2 \lfloor x/a_n\rfloor}c^2_{\lfloor x/a_n\rfloor }}{a_n|v(1-a_nz)|} \ind_{\{x> a_n \lceil y/a_n \rceil\}}.
\end{equation*}
We claim that for every fixed $x>0$, $x\neq y$, we have the pointwise convergence
\begin{equation}\label{eq:pointwise}
\lim_{n\to\infty} g_n(x) = \frac {(2|z|)^\gamma}{\Gamma(\gamma)} \e^{-2 (\Re z) x } x^{\gamma-1} \ind_{\{x>y\}}.
\end{equation}
To prove this, observe that
$$
\lim_{n\to\infty} |1-a_n z|^{2\lfloor x/a_n\rfloor}=\e^{-2(\Re z)x}.
$$
Also, by~\eqref{eq:c_k_reg_var} and the slow variation property of $L$,
$$
c^2_{\lfloor x/a_n\rfloor } = \frac{\lfloor x/a_n\rfloor^{\gamma-1}}{\Gamma(\gamma)} L(\lfloor x/a_n \rfloor) \sim \frac{x^{\gamma-1}}{a_n^{\gamma-1} \Gamma(\gamma)} L(1/a_n).
$$
Finally, by Theorem~\ref{theo:abel}, $|v(1-a_n z)| \sim |2a_n z|^{-\gamma} L(1/a_n)$. Taking everything together, we obtain the pointwise convergence stated in~\eqref{eq:pointwise}.

To complete the proof of the lemma, we have to show that the sequence $(g_n)_{n\in\N_0}$ is dominated by an integrable function. For sufficiently large $n$, we have $ \ind_{\{x> a_n \lceil y/a_n \rceil\}} \leq \ind_{\{x>y/2\}}$. Also, Theorem~\ref{theo:abel} implies that $|v(1-a_n z)| \geq  c a_n^{-\gamma} L(1/a_n)$. Furthermore, by~\eqref{eq:c_k_reg_var} we have
$$
c^2_{\lfloor x/a_n\rfloor } \leq C \frac {x^{\gamma-1}}{a_n^{\gamma-1}} L(x/a_n) \leq C \frac {x^{\gamma-1}}{a_n^{\gamma-1}} L(1/a_n) \max\{x,1/x\},
$$
provided $n$ is sufficiently large, where the last step follows from Potter's bound~\cite[Theorem~1.5.6 on p.~25]{bingham_book}.   Finally, \eqref{eq:lemma:v_with_abs_val} and the inequality
$\l(1-y\r)^\alpha \leq \e^{-\alpha y}$, which is valid for all $\alpha>0$ and $y>0$, yield
$$
|1-a_nz|^{2 \lfloor x/a_n\rfloor} \leq \left(1- \frac 12 a_n\Re z\right)^{2 \lfloor x/a_n\rfloor} \leq
\e^{-a_n (\Re z) \lfloor x/a_n \rfloor} \leq \e^{- \frac 12 (\Re z) x }.
$$
Taking all estimates together, we arrive at
$$
g_n(x) \leq C \e^{- \frac 12 (\Re z) x} x^{\gamma-1} \max \{x,1/x\} \ind_{\{x>y/2\}}
$$
for all $n\geq n_0$. Since $\Re z>0$, the dominated convergence theorem can be applied.  This completes the proof of Lemma~\ref{lemma:asymptotik_Fnk_ystat}.
\end{proof}


\vspace*{2mm}
\noindent
\emph{Tightness.} It remains to show that $(X_{a_n})_{n\in\N}$ is a tight sequence on the space $\mathcal H_\R(Q_R)$.
For random analytic functions, there are especially simple criteria of tightness. Namely, by~\cite[Remark on p.~341]{shirai}, it suffices to show that $\E|X_{a_n}(z)|^2 \leq  C$ for all $z\in Q_R$ and all sufficiently large $n\in \N$. But
\begin{align*}
\EW{|X_{a_n}(z)|^2}=\EW{X_{a_n}(z)\overline{X_{a_n}(z)}}
=
\frac{v(|1-a_n z|)}{|v(1-a_n z)|} \leq C
\end{align*}
by Lemma~\ref{lemma:v_with_abs_val}. This completes the proof of Theorem~\ref{theo:FCLT}. \hfill $\Box$

\subsection{Counting the zeroes in the Gaussian case}
In the following lemma we compute the expected number of real zeroes of the Gaussian process $Z_\gamma$, see Theorem~\ref{theo:FCLT}, in an
interval.
\begin{lemma}\label{lemma:nullstellendesgaussprozesses}
Let $N_\infty[a,b]$ be the number of real zeroes of $Z_\gamma$ in the interval $[a,b]\subset (0,\infty)$.
Then
\begin{equation*}
\E N_\infty[a,b]=\frac{\sqrt{\gamma}}{2\pi} \log \frac ba.
\end{equation*}
\end{lemma}
\begin{proof}
Recall that $Y(u):=Z(\e^u)$, $u\in\R$, is a stationary, centered Gaussian process with covariance function
$$
\rho(t) := \COV\l[Y(0),Y(t)\r] = \l(\cosh\l(\frac{t}{2}\r)\r)^{-\gamma}.
$$
Note that $\rho''(0) = -\gamma/4$.  If ${N}^*_Y[a^*,b^*]$ denotes the number of real zeroes of $Y$ in an interval $[a^*,b^*]\subset \R$, then
\begin{equation*}
\E N_\infty[a,b]=\E {N}^*_Y[\log a,\log b] = \frac{1}{\pi}\sqrt{-\rho''(0)}\log \frac ba =\frac{\sqrt {\gamma}}{2\pi} \log \frac ba
\end{equation*}
by the Rice formula; see, e.g., \cite[Thm.~7.3.2 on p.~153]{leadbetter_book}, where the formula is stated for the expected number of upcrossings which differs by a factor of $1/2$ from the expected number of zeroes.
\end{proof}

\subsection{Proof of Lemma~\ref{lem:num_zeroes_conv_distr}}
Observe that $N[1-q^{n}, 1-q^{n+1})$ is the number of zeroes of the process $X_{q^n}$ in the interval $(q,1]$. We know from Theorem~\ref{theo:FCLT} that $X_{q^n}$ converges to $Z_\gamma$ weakly on the space $\mathcal H_\R (Q_R)$, where we may take $R>1/q$, so that the interval $(q,1]$ is contained in the interior of the rectangle $Q_R$. By~\cite[Lemma~4.2]{iksanov2016}, the map which assigns to each function in $\mathcal H_\R(Q_R)$ the number of zeroes of this function in the interval $(q,1]$ is locally constant (hence, continuous) on the set of all analytic functions which do not vanish at $q,1$ and have no multiple zeroes in the interval $[q,1]$. This set has full measure w.r.t.\ the law of $Z_\gamma$ (for the a.s.\ absence of multiple zeroes, see~\cite[Lemma~4.3]{iksanov2016}), hence the continuous mapping theorem implies the weak convergence of  $N[1-q^{n}, 1-q^{n+1})$ to the number of zeroes of $Z_\gamma$ in $[q,1]$, as $n\to\infty$. The latter random variable, denoted by $N_\infty [q,1]$, takes values in $\N_0$ and has expectation $- \sqrt{\gamma} (\log q)/(2\pi)$ by Lemma~\ref{lemma:nullstellendesgaussprozesses}. \hfill $\Box$

\section{Boundedness of the expected number of zeroes}
\subsection{Decomposition of \texorpdfstring{$\E N_n$}{E Nn}}
Let $N_n := N_n[q,1] = N[1-q^{n}, 1-q^{n+1}]$ be the number of zeroes of $X_{q^n}$ in $[q,1]$, or equivalently, the number of zeroes of $f$ in $[1-q^n, 1-q^{n+1}]$.
The aim of this section is to prove Lemma~\ref{lem:num_zeroes_exp_bounded} which states that for all $1<\kappa < 2$,
$$
\sup_{n\geq n_0} \E N_n^\kappa < \infty.
$$
Because of the decomposition
\begin{equation*}
\E N_n^\kappa=\EW{N_n^\kappa\ind_{\l\{N_n\leq q^{-n^2}\r\}}}+\EW{N_n^\kappa\ind_{\l\{N_n> q^{-n^2}\r\}}},
\end{equation*}
the statement immediately follows from the following two lemmas:

\begin{lemma}\label{lemma:Nngreaterlambda}
Fix $q\in(0,1)$ and $1< \kappa < 2$. There exists $n_0\in\N$ such that
\begin{equation*}
\sup_{n\geq n_0} \EW{N_n^\kappa\ind_{\l\{N_n>q^{-n^2}\r\}}} < \infty.
\end{equation*}
\end{lemma}

\begin{lemma}\label{lemma:Nngreaterlambda2}
For all $q\in(\e^{-1/32},1)$ and $1<\kappa < 2$,
\begin{equation*}
\sup_{n\in\N} \EW{N_n^\kappa\ind_{\l\{N_n\leq q^{-n^2}\r\}}} < \infty.
\end{equation*}
\end{lemma}

\subsection{Proof of Lemma~\ref{lemma:Nngreaterlambda}}
Let us assume that $c_k\neq 0$ for all $k\in\N_0$, postponing the general case until the end of the proof.
Since $\P{\xi_0=0} \neq 1$, we can choose a sufficiently small $0<\eta<1$ such that
\begin{equation*}
p:=\P{|\xi_0|\leq \e\eta}<1.
\end{equation*}
For $k=0,1,\dots $ consider the events
\begin{equation*}
B_k:=\{|\xi_0|\leq \e \eta,\dots,|\xi_{k-1}|\leq \e \eta, |\xi_k|>\e \eta\}.
\end{equation*}
Keep in mind that
\begin{equation*}
\quad \P{B_k}=p^{k}(1-p) \quad \textrm{and}\quad \bigcup_{k\in\N_0}B_k=\Omega \mod \mathbb P,
\end{equation*}
where $(\Omega, \mathcal A, \mathbb P)$ is the probability space we are working on.
On the event $B_k$ one has
\begin{equation*}
|f^{(k)}(0)|=k!|c_k||\xi_k| \geq k!|c_k|\e \eta \geq k!|c_k|\eta.
\end{equation*}
Abbreviate $\tilde{N}_n=N[0,1-q^{n+1}]$. The theorems of Rolle and Jensen (for the latter, see, e.g.~\cite[pp.~280--281]{conway_book}) applied to $f^{(k)}$ yield on the event $B_k$ that
\begin{align}
\label{eq:estimateNnBk}
\begin{split}
N_n\leq\tilde{N}_n&\leq k+ \frac 1 {\log\l(\frac{R_n}{r_n}\r)}\log\l(\frac{\sup_{|z|=R_n}{|f^{(k)}(z)|}}{|f^{(k)}(0)|}\r) \\
&\leq k+Cq^{-n}\log\l(\frac{1}{\eta}\sum_{j=k}^\infty \binom{j}{k}\l|\frac{c_j}{c_k}\r| |\xi_j| R_n^{j-k}\r),
\end{split}
\end{align}
where we have chosen
\begin{equation*}
r_n:=1-q^{n+1}\in(0,1) \quad \textrm{and} \quad R_n:=1-q^{n+2}\in(0,1).
\end{equation*}
Besides, on $B_k$ we have
\begin{equation}\label{eq:logquadrat_concave}
\frac{1}{\eta}\sum_{j=k}^\infty \binom{j}{k}\l|\frac{c_j}{c_k}\r| |\xi_j| R_n^{j-k}
\geq \frac{|\xi_k|}{\eta}\geq  \e.
\end{equation}
On the other hand, we shall show in Lemma~\ref{lem:uniform_abel}, below, that there is $C>0$ such that
\begin{equation}\label{eq:uniform_abel}
\sum_{j=k}^{\infty} \binom{j}{k}\left|\frac{c_j}{c_k}\right|R_n^{j-k}
\leq \e^{C n(k+1)}
\end{equation}
for all $k\in\N_0$ and $n\geq n_0$. 
Using the same idea as in the standard proof of the Markov inequality leads to the estimate
\begin{multline*}
\EW{N_n^\kappa\ind_{\l\{N_n> q^{-n^2}\r\}\cap B_k}}
=
\EW{\frac{N_n^2}{N_n^{2-\kappa}}\ind_{\l\{N_n> q^{-n^2}\r\}\cap B_k}}
\\\leq
\EW{\frac{N_n^2}{q^{-n^2 (2-\kappa)}}\ind_{\l\{N_n> q^{-n^2}\r\}\cap B_k}}
\leq q^{n^2(2-\kappa)} \EW{N^2_n\ind_{B_k}}
\end{multline*}
which holds for all $k=0,1,\ldots$. The inequality
$(a+b)^2\leq 2a^2+2b^2$,
for all $a,b\in\R$, combined with~\eqref{eq:estimateNnBk} allows us to conclude that
\begin{align}\label{eq:est_one}
\EW{N_n^\kappa\ind_{\l\{N_n> q^{-n^2}\r\}\cap B_k}}
\leq Cq^{n^2(2-\kappa)}\l(k^2\P{B_k} +q^{-2n}\EW{\log^2\l(\frac{1}{\eta}\sum_{j=k}^\infty \binom{j}{k} |\xi_j|\l|\frac{c_j}{c_k}\r| R_n^{j-k}\r)\ind_{B_k}}\r).
\end{align}
Since the function $x\mapsto\log^2(x)$ is concave for $x\geq \e$ and in view of~\eqref{eq:logquadrat_concave}, we may use the inequality of Jensen (on the event $B_k$) to obtain the estimate
$$
\EW{\log^2\l(\frac{1}{\eta}\sum_{j=k}^\infty \binom{j}{k} |\xi_j|\l|\frac{c_j}{c_k}\r| R_n^{j-k}\r)\ind_{B_k}}
\leq \P{B_k} \log^2\l(\frac{1}{\P{B_k}}\EW{\frac{1}{\eta}\sum_{j=k}^\infty \binom{j}{k} |\xi_j|\l|\frac{c_j}{c_k}\r| R_n^{j-k}\ind_{B_k}}\r).
$$
Treating the term with $j=k$ separately, using the independence of $(\xi_k)_{k\in\N_0}$, the observation $\EW{|\xi_j|\ind_{B_k}} = \P{B_k} \E |\xi_1| \leq \P{B_k}$ for $j>k$, and~\eqref{eq:uniform_abel},  we obtain for sufficiently large $n$ the estimate
\begin{align}\label{eq:est_two}
\begin{split}
&\EW{\log^2\l(\frac{1}{\eta}\sum_{j=k}^\infty \binom{j}{k} |\xi_j|\l|\frac{c_j}{c_k}\r| R_n^{j-k}\r)\ind_{B_k}} \\
&\leq \P{B_k}\log^2\l(\frac{\EW{|\xi_k|\ind_{B_k}}}{\eta\P{B_k}}+\frac{1}{\P{B_k}}\EW{\frac{1}{\eta}\sum_{j=k+1}^\infty \binom{j}{k} |\xi_j|\l|\frac{c_j}{c_k}\r| R_n^{j-k}\ind_{B_k}}\r) \\
&\leq \P{B_k}\log^2\l(\frac{\EW{|\xi_k|\ind_{|\xi_k|>\e\eta}\prod_{l=0}^{k-1}\ind_{|\xi_l|\leq \e \eta}}}{\eta p^k(1-p)}+\frac{1}{\eta}\sum_{j=k+1}^\infty \binom{j}{k} \l|\frac{c_j}{c_k}\r| R_n^{j-k}\r) \\
&\leq \P{B_k}\log^2\l(\frac{\EW{|\xi_k|\ind_{|\xi_k|>\e\eta}}}{\eta (1-p)}+\frac{1}{\eta}\e^{Cn(k+1)}\r) \\
&\leq \P{B_k}\log^2\l(\frac{2}{\eta}\e^{Cn(k+1)}\r) \\
&\leq C\P{B_k} n^2(k+1)^2.
\end{split}
\end{align}
Taking together the above estimates \eqref{eq:est_one} and \eqref{eq:est_two}, we obtain
\begin{equation*}
\EW{N_n^\kappa\ind_{\l\{N_n>q^{-n^2}\r\}}\ind_{B_k}}
\leq Cq^{n^2(2-\kappa)}\l(k^2\P{B_k}+ q^{-2n}\P{B_k} n^2(k+1)^2\r).
\end{equation*}
Therefore, taking the sum over $k=0,1,\ldots$, we arrive at
\begin{align*}
\EW{N_n^\kappa\ind_{\l\{N_n>q^{-n^2}\r\}}}
&=\sum_{k=0}^\infty \EW{N_n^\kappa\ind_{\l\{N_n>q^{-n^2}\r\}}\ind_{B_k}}\\
&\leq C\l(q^{n^2(2-\kappa)}\sum_{k=0}^\infty k^2\P{B_k} + q^{n^2(2-\kappa)/2}\sum_{k=0}^\infty (k+1)^2 \P{B_k} \r)\\
&\leq C\l(\sum_{k=0}^\infty k^2 p^k(1-p) + \sum_{k=0}^\infty (k+1)^2
p^k(1-p) \r),
\end{align*}
which is a finite constant. This completes the proof of Lemma~\ref{lemma:Nngreaterlambda} in the case when $c_k\neq 0$ for all $k\in\N_0$.
Let us stress that in the above proof $q^{-n^2}$ could be replaced by, say, $q^{-An}$ with sufficiently large $A>0$.

Let us finally consider the case in which some of the $c_k$'s may vanish. It follows from~\eqref{eq:c_k_reg_var} that $c_k\neq 0$ for all but finitely many $k$. Let $K\in\N_0$ be such that $c_k\neq 0$ for $k\geq K$. By Rolle's theorem, $N_n \leq N_n' + K$, where $N_n'$ is the number of zeroes of $f^{(K)}$, the $K$-th derivative of $f$, in the interval $[1-q^n, 1-q^{n+1}]$. It follows that
\begin{equation}\label{eq:N_n_vs_N_n_prime}
\EW{N_n^\kappa \ind_{\l\{N_n > q^{-n^2}\r\}}} \leq \EW{ (N_n' + K)^\kappa \ind_{\l\{N_n'+K > q^{-n^2}\r\}}} \leq C K^\kappa + C\EW{(N_n')^\kappa \ind_{\l\{N_n'> q^{-n^2}/2\r\}}}
\end{equation}
if $n$ is sufficiently large, where in the last step we used the inequality $(a+b)^\kappa \leq 2^{\kappa-1}(a^\kappa + b^\kappa)$. Now,
$$
f^{(K)}(z) = \sum_{k=0}^\infty \xi_{k+K} c_{k+K}  (k+K)(k+K-1)\ldots (k+1) z^k =:  \sum_{k=0}^\infty \xi_{k+K} c_k' z^k
$$
has the same form as $f$, and the coefficients $c_k'$ satisfy~\eqref{eq:c_k_reg_var} with $\gamma$ replaced by $\gamma+2K$, while being non-zero. Applying the above proof to $f^{(K)}$ with $q^{-n^2}$ replaced by $q^{-n^2}/2$, we obtain
$$
\sup_{n\geq n_0} \EW{(N_n')^\kappa \ind_{\l\{N_n'> q^{-n^2}/2\r\}}} < \infty.
$$
Together with~\eqref{eq:N_n_vs_N_n_prime}, this yields the required statement. \hfill $\Box$

\subsection{Proof of Lemma~\ref{lemma:Nngreaterlambda2}}
In the following we may assume that $n$ is sufficiently large.
Fix $q\in(\e^{-1/32},1)$ and $1<\kappa<2$. We are going to show that
\begin{equation*}
\sup_{n\in\N} \EW{N_n^\kappa\ind_{\l\{N_n\leq q^{-n^2}\r\}}} < \infty.
\end{equation*}
For $m=0,1,\ldots$ let  $D_m^{(n)}$ denote the event that $X_{q^n}$ has at least $m$ zeroes in the interval $[q,1]$.
Keep in mind that $D^{(n)}_0 \supset D^{(n)}_1\supset \ldots$.
The following crucial estimate will be stated and proved in Lemma~\ref{lemma:abschaetzungdmn} below:
\begin{equation}\label{eq:lemma:abschaetzungdmn}
\P{D^{(n)}_m}\leq C\l(\l(\frac{2q}{3-q}\r)^{2m/3}+\l(\frac{2q}{3-q}\r)^{-m/3}\EXP{-\frac{n^2}{8}}\r).
\end{equation}
For every sufficiently large $n\in\N$ there exists a $k_0=k_0(n,q)\in \N$ such that
\begin{equation}\label{eq:def_knull}
\EXP{\frac{n^2}{32}}\leq \l(\frac{2q}{3-q}\r)^{-k_0/3} \leq \EXP{\frac{n^2}{16}}.
\end{equation}
From \eqref{eq:def_knull} one has
\begin{equation}\label{eq:def_knull_folgerung}
\l(\frac{2q}{3-q}\r)^{2k_0/3} \leq \EXP{-\frac{n^2}{16}}.
\end{equation}
Using the inclusions $D^{(n)}_0 \supset D^{(n)}_1\supset \ldots$ we have
\begin{multline*}
\EW{N_n^\kappa\ind_{\l\{N_n\leq q^{-n^2}\r\}}}
\leq \EW{N_n^2\ind_{\l\{N_n\leq q^{-n^2}\r\}}}
= \sum_{k=1}^{\lfloor  q^{-n^2} \rfloor} (2k-1)\P{D_k^{(n)}} \\
\leq \sum_{k=1}^{k_0}(2k-1)\P{D_k^{(n)}}
+\sum_{k=k_0+1}^{\lfloor q^{-n^2} \rfloor} (2k-1)\P{D^{(n)}_{k_0}}.
\end{multline*}
Note that $2q/(3-q)\leq 1$. Applying~\eqref{eq:lemma:abschaetzungdmn},\eqref{eq:def_knull} and \eqref{eq:def_knull_folgerung} to the right-hand side, we obtain
\begin{align*}
\EW{N_n^\kappa\ind_{\l\{N_n\leq q^{-n^2}\r\}}} &\leq C \l(\sum_{k=1}^{k_0}(2k-1)\l(\frac{2q}{3-q}\r)^{2k/3}+\EXP{-\frac{n^2}{8}}\sum_{k=1}^{k_0}(2k-1)\l(\frac{2q}{3-q}\r)^{-k/3}\r) \\
&\qquad  +C q^{-2n^2}\l(\l(\frac{2q}{3-q}\r)^{2k_0/3}+\EXP{-\frac{n^2}{8}}\l(\frac{2q}{3-q}\r)^{-k_0/3}\r) \\
&\leq C \l(\sum_{k=1}^{\infty}(2k-1)\l(\frac{2q}{3-q}\r)^{2k/3}+\EXP{-\frac{n^2}{16}} k_0^2+2q^{-2n^2}\EXP{-\frac{n^2}{16}}\r),
\end{align*}
which is bounded above by $C$ since $\e^{-1/32}< q<1$ by assumption and $k_0< Cn^2$ by~\eqref{eq:def_knull}. This completes the proof of Lemma~\ref{lemma:Nngreaterlambda2} modulo the estimate~\eqref{eq:lemma:abschaetzungdmn}. \hfill $\Box $


\subsection{Probability that \texorpdfstring{$f$}{f} has at least \texorpdfstring{$m$}{m} zeroes in \texorpdfstring{$[1-q^n, 1-q^{n+1}]$}{[1-qn, 1-q(n+1)]}}
In this section we prove the estimate for the probability of the event $D_m^{(n)}$ that $X_{q^n}$ has at least $m$ zeroes in the interval $[q,1]$. This estimate was used in the proof of Lemma~\ref{lemma:Nngreaterlambda2}.
\begin{lemma}\label{lemma:abschaetzungdmn}
Let $q\in(-2+\sqrt{7},1)$.  There exist constants $C>0$, $n_0\in\N$ such that for all $n\geq n_0$ and $m\in\N_0$,
\begin{equation*}
	\P{D^{(n)}_m}\leq C\l(\l(\frac{2q}{3-q}\r)^{2m/3}+\l(\frac{2q}{3-q}\r)^{-m/3}\EXP{-\frac{n^2}{8}}\r).
\end{equation*}
\end{lemma}
For the proof we need two auxiliary lemmas.
The first of them is essentially contained in the paper of Ibragimov and Maslova~\cite{ibragimov_maslova1}, but since they stated it only in some special case, we give a full proof.
\begin{lemma}\label{lemma:abschaetzungdmzwei}
Let $(Q(t))_{t\in [\alpha,\beta]}$ be a stochastic process whose sample paths are $m$ times continuously differentiable with probability $1$. Let $D_m$ denote the event that
$Q(t)$ has at least $m\in\N$ zeroes on $[\alpha,\beta]$. Then the estimate
\begin{equation*}
\P{D_m\cap\left\{\left|Q(\beta)\right|\geq T\right\}}\leq \left[\frac{(\beta-\alpha)^m}{m!T}\right]^2
\sup_{x\in[\alpha,\beta]} \E\l|\frac{\d^mQ}{\d x^m}(x)\r|^2
\end{equation*}
holds for all $T>0$.
\end{lemma}
\begin{proof}
By Rolle's theorem, on the event $D_m$ we can find (random) points $t_0\geq \ldots \geq t_{m-1}$ in the interval $[\alpha,\beta]$ such that
\begin{equation*}
\frac{\d^jQ}{\d x^j}(t_j)=0 \text{ for all } j\in\{0,\dots, m-1\}.
\end{equation*}
Thus, we may consider the random variable
\begin{equation*}
Y:=\ind_{D_m}\times
\INT{t_0}{\beta}\INT{t_1}{x_1}\dots \INT{t_{m-1}}{x_{m-1}} \frac{\d^m Q}{\d x^m}(x_m)\D x_m\dots \d x_1.
\end{equation*}
On the event $D_m$, the random variables $Q(\beta)$ and $Y$ are equal. On the complement of $D_m$, we have $Y=0$.  Hence, it follows that
\begin{equation*}
\P{D_m\cap\left\{|Q(\beta)|\geq T\right\}}\leq \P{|Y|\geq T}.
\end{equation*}
Markov's inequality yields
\begin{align*}
\P{|Y|\geq T}
\leq\frac{1}{T^2}\E
\left|\int_{t_0}^{\beta}\int_{t_1}^{x_1}\dots \int_{t_{m-1}}^{x_{m-1}}\frac{\d^m Q}{\d x^m}(x_m)\D x_m\dots\d x_1\right|^2.
\end{align*}
Applying the Cauchy--Schwarz inequality to the multiple integral  yields the estimate
\begin{align*}
\P{|Y|\geq T}&\leq \frac{1}{T^2}\frac{(\beta-\alpha)^m}{m!}
\E
\int_{t_0}^{\beta}\int_{t_1}^{x_1} \dots \int_{t_{m-1}}^{x_{m-1}}\l|\frac{\d^m Q}{\d x^m}(x_m)\r|^2\D x_m\dots \d x_1  \\
&\leq
\left[\frac{(\beta-\alpha)^m}{m!T}\right]^2
\sup_{x\in[\alpha,\beta]} \mathbb{E}\l|\frac{\d^m Q}{\d x^m}(x_m)\r|^2,
\end{align*}
where in the second inequality we interchanged the integral and the expectation and estimated the integrand by its maximum.
\end{proof}

\begin{lemma}\label{lemma:abschaetzungrandomanalyticfunktion} Let $q\in(-2+\sqrt{7},1)$. For all $n\geq n_0$ and  $m\in\N_0$ we have
	\begin{equation*}
	\sup_{x\in [1-q^n,1-q^{n+1}]}\E\l|\frac{\d^m f}{\d x^m}(x)\r|^2 \leq
	\frac{3(m!)^2q^{-2}v\l(1-\alpha(q) q^n\r)}{(q^n-q^{n+1})^{2m}(3/(2q)-1/2)^{2(m+1)}},
	\end{equation*}
	where $\alpha(q)$ is given by
	\begin{equation*}
	\alpha(q):=2+\frac{q}{2}-\frac{3}{2q}.
	\end{equation*}
\end{lemma}
\begin{proof}
For $n\in\N$ and $q\in(-2+\sqrt{7},1)$ put
\begin{align*}
z_n:= 1-\frac{q^n+q^{n+1}}{2}, \quad
r_n:= \frac{3(q^{n}-q^{n+1})}{2q},  \quad
\delta_n:=\l(\frac{3}{2q}-\frac{1}{2}\r)\l(q^n-q^{n+1}\r).
\end{align*}
Note that $z_n$ is the middle point of the interval $[1-q^n,1-q^{n+1}]$, and
$$
\delta_n = z_n + r_n - (1-q^{n+1})
= 1-q^n - (z_n - r_n).
$$
For $m\in\N_0$ let $f^{(m)}$ denote the $m$-th derivative of $f$. Let $B_{r_n}(z_n)$ be the disk centered at $z_n$ and having radius $r_n$. Also, denote by $\partial B_{r_n}(z_n)$ its boundary. Cauchy's integral formula for analytic functions yields for all $x\in [1-q^n,1-q^{n+1}]$ the estimate
\begin{equation*}
|f^{(m)}(x)|=
\frac{m!}{2\pi}
\l|\oint_{\partial B_{r_n}(z_n)} \frac{f(z)}{(z-x)^{m+1}}\D z \r|
\leq\frac{m!}{2\pi}\int_{\partial B_{r_n}(z_n)} \frac{|f(z)|}{\delta^{m+1}_n}\ |\d z|.
\end{equation*}
After squaring, taking the expectation and using Jensen's inequality for the quadratic function, we obtain
\begin{align*}
\E |f^{(m)}(x)|^2
&\leq \frac{(m!)^2}{4\pi^2} \frac{1}{\delta_n^{2(m+1)}}\E \left(\int_{\partial B_{r_n}(z_n)} |f(z)|\ |\d z|\right)^2 \\
&\leq \frac{(m!)^2}{2\pi}\frac{r_n}{\delta_n^{2(m+1)}}\E \int_{\partial B_{r_n}(z_n)} |f(z)|^2\ |\d z| \\
&= \frac{(m!)^2}{2\pi}\frac{r_n}{\delta_n^{2(m+1)}}\int_{\partial B_{r_n}(z_n)} \E |f(z)|^2\ |\d z| \\
&\leq (m!)^2\frac{r_n^2}{\delta_n^{2(m+1)}}\sup_{z\in \partial B_{r_n}(z_n)} \E |f(z)|^2 \\
& =  (m!)^2 \frac{r_n^2}{\delta_n^{2(m+1)}} v(z_n+r_n),
\end{align*}
where the last step follows from the definition of $v$; see~\eqref{eq:def_v}.
Recalling the definitions of $r_n,\delta_n, z_n$ and that $x\in [1-q^n,1-q^{n+1}]$ was arbitrary, we arrive at
\begin{align*}
\sup_{x\in [1-q^n,1-q^{n+1}]} \E |f^{(m)}(x)|^2
&\leq 3(m!)^2 \frac{q^{-2}v\l(1-\frac{q^n+q^{n+1}}{2}+\frac{3(q^n-q^{n+1})}{2q}\r)}{(q^n-q^{n+1})^{2m}(3/(2q)-1/2)^{2(m+1)}} \\
&\leq 3(m!)^2 \frac{q^{-2}v\left(1-q^n\left(\frac{1+q}{2}-\frac{3(1-q)}{2q}\right)\right)}{(q^n-q^{n+1})^{2m}(3/(2q)-1/2)^{2(m+1)}} \\
&\leq 3(m!)^2 \frac{q^{-2}v\l(1-\alpha(q) q^n\r)}{(q^n-q^{n+1})^{2m}(3/(2q)-1/2)^{2(m+1)}}.
\end{align*}
\end{proof}

\begin{proof}[Proof of Lemma~\ref{lemma:abschaetzungdmn}]
Recall that $X_{q^n}(1)$ was defined in \eqref{eq:defXn}. For $T>0$ we can write
\begin{equation*}\label{eq:zerlegungPDMJ}
\P{D^{(n)}_m}\leq \P{D^{(n)}_m\cap\left\{\left|X_{q^n}(1)\right|\geq T\right\}} + \P{\left|X_{q^n}(1)\right|< T}.
\end{equation*}
The first term of the sum can be estimated using Lemmas \ref{lemma:abschaetzungdmzwei}
and \ref{lemma:abschaetzungrandomanalyticfunktion} as follows:
\begin{align*}
&\P{D^{(n)}_m\cap\left\{\left|X_{q^n}(1)\right|\geq T\right\}} \\
&\quad =\P{D^{(n)}_m\cap\left\{\left|f(1-q^n)\right|\geq T\sqrt{|v(1-q^n)|}\right\}} \\
&\quad=\frac{1}{|v(1-q^n)|}\l[\frac{(q^{n}-q^{n+1})^m}{m! T}\r]^2 \sup_{x\in[1-q^n,1-q^{n+1}]}\E\l|\frac{\d^m f}{\d x^m}(x)\r|^2 \\
&\quad\leq
C\frac{v(1-\alpha(q)q^n)}{v(1-q^n)}
\frac{q^{-2}}{T^2}\l(\frac{3}{2q}-\frac{1}{2}\r)^{-2(m+1)} \\
&\quad\leq
C\frac{q^{-2}}{T^2}\l(\frac{3}{2q}-\frac{1}{2}\r)^{-2(m+1)} \\
&\quad\leq\frac{C}{T^2}\l(\frac{2q}{3-q}\r)^{2m}.
\end{align*}
Here, we used the following estimate which is a consequence of Theorem~\ref{theo:abel}:
$$
\frac{v(1-\alpha(q)q^n)}{v(1-q^n)} \leq C.
$$
The probability that $|X_{q^n}(1)|< T$ will be estimated in Lemma~\ref{lemma:abschaetzungdmeins} that we shall prove in Section~\ref{sec:lemma_5_1}. Taking these estimates together, we obtain
\begin{equation*}
\P{D^{(n)}_m}\leq C\l(\frac{1}{T^2}\l(\frac{2q}{3-q}\r)^{2m}+ T+T^{-1/2}\EXP{-\frac{n^2}{8}}\r).
\end{equation*}
Now we choose $T=(2q/(3-q))^{2m/3}$ to obtain the statement of Lemma~\ref{lemma:abschaetzungdmn}.
\end{proof}

\section{Probability of small values of \texorpdfstring{$f$}{f}} \label{sec:lemma_5_1}
In this section we estimate the probability of the event $|X_{q^n}(1)|< T$, $T>0$, which was a  crucial ingredient in  the proof of Lemma~\ref{lemma:abschaetzungdmn}.  Recall that
$$
X_{q^n} (1) = \frac {f(1-q^n)}{\sqrt{v(1-q^n)}} = \sum_{k=0}^\infty \xi_k \frac{c_k (1-q^n)^k}{\sqrt{v(1-q^n)}} = \sum_{k=0}^\infty a_{n,k} \xi_k,
$$
where the array $\{a_{n,k}:k\in\N_0, n\in\N\}$ is given by
\begin{equation}\label{eq:def_ank}
a_{n,k}:=\frac{c_k(1-q^n)^{k}}{\sqrt{v(1-q^n)}} \in\R.
\end{equation}

\begin{lemma} \label{lemma:abschaetzungdmeins} Let $q\in(0,1)$ be fixed. There exists a constant $C>0$ such that
\begin{equation*}
\P{\left|X_{q^n}(1)\right|\leq T}
\leq C\l(T+T^{-1/2}\EXP{-\frac{n^2}{8}}\r)
\end{equation*}
for all $n\geq n_0$ and all $T>0$.
\end{lemma}
\begin{proof}
For $\lambda>0$ consider the random variable
\begin{equation*}
\tilde{X}_{n,\lambda}:= \Theta_\lambda + X_{q^n}(1) = \Theta_\lambda + \sum_{k=0}^\infty a_{n,k} \xi_k,
\end{equation*}
where $\Theta_\lambda$ is the sum of two independent random variables that are uniformly distributed on $[-\lambda,\lambda]$ and also independent of $X_{q^n}(1)$. The characteristic function of  $\Theta_\lambda$ is denoted by
\begin{equation*}
\psi_\lambda(t):=\EW{\e^{\mathrm{i} t \Theta_\lambda}}=\frac{\sin^2(t\lambda)}{t^2\lambda^2}.
\end{equation*}
We have
\begin{equation}\label{eq:zerlegunglemmazehn}
\P{\left|X_{q^n}(1)\right|\leq T}\leq \P{\l|\tilde{X}_{n,\lambda}\r|\leq \frac{3}{2}T}
+\P{|\Theta_\lambda|\geq \frac{1}{2}T}.
\end{equation}

\vspace*{2mm}
\noindent
\emph{First term on the right-hand side of \eqref{eq:zerlegunglemmazehn}.}
Let $\tilde{\phi}_{n,\lambda}$ denote the characteristic function of $\tilde{X}_{n,\lambda}$, that is
$$
\tilde{\phi}_{n,\lambda} (t) = \EW{\e^{\I t \tilde{X}_{n,\lambda}}} = \psi_\lambda(t) \prod_{k=0}^\infty \varphi(a_{n,k} t),
$$
where $\varphi(t) = \E{\e^{\I t \xi_0}}$ is the characteristic function of the $\xi_k$'s.
The density of $\tilde{X}_{n,\lambda}$ exists and for $y\geq 0$  we can use Fourier inversion to represent the distribution function of $|\tilde{X}_{n,\lambda}|$ as
\begin{align*}
\P{\l|\tilde{X}_{n,\lambda}\r|\leq y}
&=\frac{1}{2\pi}\INT{-y}{y}\INT{-\infty}{\infty} \tilde{\phi}_{n,\lambda}(t) \e^{-\I t x}\D t \D x
 =\frac{2}{\pi}\INT{0}{\infty} \frac{\sin(yt)}{t}
\RE \tilde{\phi}_{n,\lambda}(t) \D t\\
&\leq
\frac{2y}{\pi} \INT{0}{\infty} \psi_\lambda(t)
\prod_{k=0}^\infty\l|\phi\left(a_{n,k} t
\right)\r|
\D t.
\end{align*}
In the last inequality we used the estimates $|\sin (yt)| \leq yt$ and $|\RE \tilde{\phi}_{n,\lambda}(t)|\leq |\tilde{\phi}_{n,\lambda}(t)|$.

Observe that for every $n\in\N_0$ the sequence $(a_{n,k}^2)_{k\in\N_0}$ defines a probability distribution
on $\N_0$, namely $\sum_{k=0}^\infty a_{n,k}^2 = 1$.  Let $b_{n,k}:=a_{n,(k)}$
denote a descending rearrangement of $(a_{n,k})_{k\in\N_0}$, that is
$$
\{b_{n,0}, b_{n,1},\ldots\} = \{a_{n,0}, a_{n,1},\ldots\} \text{ and } b_{n,0}\geq b_{n,1}\geq \ldots.
$$
Then, according to Lemma \ref{lemma:ankmaximum} below, we have
\begin{equation}\label{eq:bnnull_est}
 b_{n,0}^2=\max_{k=0,1,\dots}a_{n,k}^2 \leq q^{n/2 \cdot(1\wedge \gamma)}.
\end{equation}
Since  the coefficients $(\xi_k)_{k\in\N_0}$ are supposed to have zero mean and unit variance,
there exists a constant $\eta>0$ for which their characteristic function $\phi$ satisfies
\begin{equation}\label{eq:est_argphi}
|\phi(t)|\leq \EXP{-\frac{t^2}{4}} \quad \textrm{for all } t\in[-\eta,\eta].
\end{equation}
We have the estimate
\begin{equation*}
\P{\l|\tilde{X}_{n,\lambda}\r|\leq y} \leq  \frac{2y}{\pi}\sum_{k=0}^{\infty} \int_{\Gamma_k} \psi_\lambda(t)\prod_{j=0}^\infty \l|\phi(b_{n,j}t)\r|\D t,
\end{equation*}
where $\{\Gamma_k:k\in\N_0\}$ is a disjoint partition of $[0,\infty)$ given by
\begin{equation*}
\Gamma_0:=\l[0,\frac{\eta}{b_{n,0}}\r),\quad
\Gamma_k:=\l[\frac{\eta}{b_{n,k-1}},\frac{\eta}{b_{n,k}}\r), \quad k=1,2,\dots.
\end{equation*}
For $t\in\Gamma_0$ we have $\max_{j=0,1,\ldots} b_{n,j}t  = b_{n,0}t \leq \eta$ and therefore
\begin{equation}\label{eq:est_Inull}
\int_{\Gamma_0}\psi_\lambda(t)\prod_{j=0}^\infty \l|\phi(b_{n,j} t)\r| \D t\leq\INT{0}{\eta b_{n,0}^{-1}}
\EXP{-\frac{t^2}{4}} \leq C.
\end{equation}
The idea of the following is that on $\Gamma_k$ with $k=1,2,\dots$, the arguments of $\phi$ in the first $k$  factors (corresponding to $j=0,\ldots,k-1$) of the product $\prod_{j=0}^\infty |\phi(b_{n,j}t)|$ are ``too big'' and $\phi$ has to be estimated in a trivial way by $|\phi|\leq 1$, while the remaining factors can be estimated by means of~\eqref{eq:est_argphi}.

The idea of the above estimates is due to \citet{ibragimov_maslova1}. However, in their paper the weights $c_k$ are equal to $1$ (and the summation in the definition of $f$ stops at $n$). In our more general situation, we had to introduce the rearrangement $(b^2_{n,k})_{k\in\N_0}$ of $(a^2_{n,k})_{k\in\N_0}$ to make the argument work. In what follows, we shall strongly rely on non-trivial estimates of the order statistics $(b^2_{n,k})_{k\in\N_0}$ which will be collected in Section~\ref{sec:rearrangement}.

Denote the right tails of $(a^2_{n,k})_{k\in\N_0}$ and $(b^2_{n,k})_{k\in\N_0}$ by
\begin{equation}\label{eq:def_FnkandFtildenk}
\tilde{F}_{n,k}:=\sum_{j=k}^\infty a_{n,j}^2\quad \text{and} \quad F_{n,k}:=\sum_{j=k}^\infty
b_{n,j}^2,\quad k=0,1,2,\dots.
\end{equation}
Observe that $F_{n,0} = \tilde F_{n,0} = 1$, both sequences $(F_{n,k})_{k\in\N_0}$ and $(\tilde F_{n,k})_{k\in\N_0}$ are non-increasing, converge to $0$, and
\begin{equation}\label{eq:FnkkleinerFnktilde}
F_{n,k}\leq \tilde{F}_{n,k}\quad \text{for all }k\in\N_0, n\in \N.
\end{equation}
For $k\in\N$ the following estimate holds:
\begin{align*}
\int_{\Gamma_k} \psi_\lambda(t)\prod_{j=0}^\infty \l|\phi(b_{n,j}t)\r|\D t
&\leq\INT{\eta b_{n,k-1}^{-1}}{\infty}\frac{1}{t^2\lambda^2}\EXP{-\frac{t^2}{4}F_{n,k}} \D t \\
& = \frac{\sqrt{F_{n,k}}}{\lambda^2}\INT{\eta\sqrt{F_{n,k}} b_{n,k-1}^{-1}}{\infty}
\frac{1}{u^2}\EXP{-\frac{u^2}{4}}\D u \\
&\leq C \frac{\sqrt{F_{n,k}}}{\lambda^2}\EXP{-\nu\frac{F_{n,k}}{b_{n,k-1}^2}},
\end{align*}
where we used the change of variables $u = t\sqrt{F_{n,k}}$ and  $\nu>0$ is a small constant.
Let $k_0(n)$ and $d_0(n)$ be chosen in such a way that for every $n\in\N$
\begin{align*}
1=F_{n,0}\geq \dots \geq F_{n,k_0(n)-1}\geq F_{n,k_0(n)} \geq q^{n/4(1\wedge \gamma)} >
F_{n,k_0(n)+1}\geq \ldots, \\
1=\tilde{F}_{n,0}\geq \dots \geq \tilde{F}_{n,d_0(n)-1}\geq \tilde{F}_{n,d_0(n)} \geq q^{n/4(1\wedge
\gamma)} > \tilde{F}_{n,d_0(n)+1}\geq \ldots.
\end{align*}
For $n\in\N$ and $k\in\N_0$ define
\begin{equation*}
I_{n,k}:=\sqrt{F_{n,k}}\EXP{-\nu \frac{F_{n,k}}{b^2_{n,k-1}}}.
\end{equation*}
We are interested in estimating the sum
\begin{equation}\label{eq:drei_Summen}
\sum_{k=1}^{\infty}I_{n,k} =
\sum_{k=1}^{k_0(n)}I_{n,k}
+\sum_{k=k_0(n)+1}^{\lfloor n^2q^{-n}\rfloor}I_{n,k}
+\sum_{k=\lceil n^2q^{-n}\rceil}^{\infty}I_{n,k}.
\end{equation}

\vspace*{2mm}
\noindent
\emph{First sum on the right-hand side of \eqref{eq:drei_Summen}.}
By~\eqref{eq:FnkkleinerFnktilde} we have
\begin{equation}\label{eq:k0_groesser_d_0}
 k_0(n)\leq d_0(n)\quad \text{for all }n\in\N.
\end{equation}
If $\tilde{F}_{n,k} < q^{n/4(1\wedge\gamma)}$ then $k>d_0(n)$.
Lemma~\ref{lemma:asymptotik_Fnk}(i), below,   implies that for a sufficiently large constant $A>0$,
\begin{equation}\label{eq:bedanA}
 \tilde{F}_{n, \lfloor Anq^{-n}\rfloor}\leq C \e^{-An} < q^{n/4(1\wedge \gamma)}\quad \textrm{for all } n\in\N,
\end{equation}
which together with \eqref{eq:k0_groesser_d_0} implies an upper bound on $k_0(n)$, namely
\begin{equation}\label{eq:upboundfordk0}
k_0(n)\leq d_0(n)\leq A nq^{-n} \quad \textrm{for all }n\in \N.
\end{equation}
Using \eqref{eq:upboundfordk0} and \eqref{eq:bnnull_est},  we find
\begin{align}\label{eq:est_erstersummand}
\begin{split}
\sum_{k=1}^{k_0(n)} I_{n,k} &\leq \sum_{k=1}^{k_0(n)} \EXP{-\nu\frac{F_{n,k}}{b_{n,k-1}^2}}
\leq \sum_{k=1}^{k_0(n)} \EXP{-\nu\frac{q^{n/4(1\wedge\gamma)}}{b_{n,0}^2}} \\
&\leq C nq^{-n} \EXP{-\nu q^{-n/4(1\wedge\gamma)}}.
\end{split}
\end{align}

\vspace*{2mm}
\noindent
\emph{Second sum on the right-hand side of \eqref{eq:drei_Summen}}.
At first we prove a suitable lower bound on $k_0(n)$. If $F_{n,k}\geq q^{n/4(1\wedge \gamma)}$ then $k\leq k_0(n)$.
We claim that for sufficiently small $\delta>0$ there exists $n_0\in\N$ such that
\begin{equation}\label{eq:untSchrankefuerk0}
k_0(n)\geq n^\delta q^{-n} \quad \textrm{for all }n\geq n_0.
\end{equation}
To prove~\eqref{eq:untSchrankefuerk0}, observe that it follows from Lemma \ref{lemma:est_bnk_cor2}, below,  that for every $0<\delta_1<\delta_2<1$ there exists $n_0\in\N$ such that
\begin{equation*}
F_{n,\lfloor n^{\delta_1}q^{-n}\rfloor}\geq \tilde{F}_{n,\lfloor n^{\delta_2}q^{-n}\rfloor} \quad \textrm{for all } n\geq n_0.
\end{equation*}
Now choose sufficiently small $\delta_2>\delta_1>0$ and $\eps>0$ such that
\begin{equation*}
F_{n,\lfloor n^{\delta_1}q^{-n}\rfloor }\geq \tilde{F}_{n,\lfloor n^{\delta_2}q^{-n}\rfloor}\geq c q^{n\eps}n^{\delta_2(\gamma-1-\eps)}\e^{-2n^{\delta_2}}\geq q^{n/4(1\wedge \gamma)},
\end{equation*}
for sufficiently large $n$, where in the second estimate we used Lemma \ref{lemma:asymptotik_Fnk}(ii), below.
Setting $\delta:=\delta_1>0$ this implies \eqref{eq:untSchrankefuerk0}.

Take some $\alpha \in (\delta, 1)$.  By Lemmas~\ref{lemma:est_bnk_cor1} and~\ref{lemma:est_bnk_cor2}, below, there exist $n_0\in\N$ and $\zeta>0$
such that
\begin{align}\label{eq:est_bnk_appli}
\begin{split}
b^2_{n,k-1} &\leq  a^2_{n,\lfloor \zeta q^{-n}\rfloor + k - \lfloor n^\delta q^{-n}\rfloor }\quad \textrm{and} \\
F_{n,k}&\geq \tilde{F}_{n,k + \lfloor n^{\alpha} q^{-n}\rfloor}  \quad \textrm{for all }n\geq n_0.
\end{split}
\end{align}
Using~\eqref{eq:est_bnk_appli} and Lemma~\ref{lemma:asymptotik_Fnk}(ii), below,  we obtain for every $n^\delta q^{-n}\leq k\leq n^2 q^{-n}$ and $0<\eps<1$ the estimate
\begin{align}\label{eq:quot_Fnkbnk_est}
\begin{split}
\frac{F_{n,k}}{b^2_{n,{k-1}}}&\geq
\frac{\tilde{F}_{n,k + \lfloor n^{\alpha}q^{-n}\rfloor}}{a^2_{n,k + \lfloor \zeta q^{-n} \rfloor -  \lfloor n^\delta q^{-n}\rfloor }} \\
&\geq\frac{Cq^{n\eps}(kq^n+n^\alpha)^{\gamma-1-\eps}\EXP{-2(kq^{n}+n^\alpha)}}{q^{n(\gamma-\eps)} (k+\zeta q^{-n}-n^\delta q^{-n})^{\gamma-1+\eps}(1-q^n)^{2(kq^n-n^\delta)q^{-n}}} \\
&\geq\frac{Cq^{n\eps}(kq^n+n^\alpha)^{\gamma-1-\eps}\EXP{-2(kq^{n}+n^\alpha)}}{q^{n(\gamma-\eps)}q^{-n(\gamma-1+\eps)} (kq^n+\zeta-n^\delta)^{\gamma-1+\eps}(1-q^n)^{2(kq^n-n^\delta)q^{-n}}} \\
&\geq C q^{-n(1-3\eps)}
\l(\frac{kq^n+n^\alpha}{kq^n+\zeta-n^\delta}\r)^{\gamma-1} \frac{\EXP{-2(n^\alpha-n^\delta)}}{(kq^n+n^\alpha)^{\eps}(kq^n+\zeta -n^\delta)^{\eps}}\\
&\geq C q^{-n(1-3\eps)/2}.
\end{split}
\end{align}
Taking \eqref{eq:untSchrankefuerk0}, \eqref{eq:est_bnk_appli} and \eqref{eq:quot_Fnkbnk_est} into account, the second sum in~\eqref{eq:drei_Summen} can be estimated in the following way:
\begin{align}\label{eq:est_zweitersummand}
\begin{split}
\sum_{k=k_0(n)+1}^{\lfloor n^2q^{-n}\rfloor}I_{n,k}
&\leq C\sum_{k = \lfloor n^{\delta}q^{-n} \rfloor}^{\lfloor n^2q^{-n}\rfloor}\EXP{-\nu\frac{F_{n,k}}{b^2_{n,{k-1}}}} \\
&\leq C
\sum_{k=\lfloor n^{\delta}q^{-n}\rfloor }^{\lfloor n^2q^{-n}\rfloor }\EXP{-\nu q^{-n(1-3\eps)/2}} \\
&\leq C
n^2q^{-n}\EXP{-\nu q^{-n(1-3\eps)/2}},
\end{split}
\end{align}
where $\nu>0$ is a small constant.

\vspace*{2mm}
\noindent
\emph{Third term on the right-hand side of \eqref{eq:drei_Summen}}.
We use \eqref{eq:FnkkleinerFnktilde} and again Lemma \ref{lemma:asymptotik_Fnk} (i), below, to obtain
\begin{align}\label{eq:est_drittersummand}
\sum_{k=\lceil n^2q^{-n}\rceil }^\infty I_{n,k}
\leq \sum_{k=\lceil n^2q^{-n}\rceil}^\infty \sqrt{F_{n,k}}
\leq \sum_{k= \lceil n^2q^{-n} \rceil }^\infty \sqrt{\tilde{F}_{n,k}}
\leq C\sum_{k= \lceil n^2q^{-n}\rceil }^\infty \e^{-kq^n/2}
\leq C e^{-n^2/4}.
\end{align}
Taking together the results of~\eqref{eq:est_Inull} and~\eqref{eq:est_erstersummand} as well as~\eqref{eq:est_zweitersummand} and~\eqref{eq:est_drittersummand} yields
\begin{equation}\label{eq:firsttermontherhs}
\P{\l|\tilde{X}_{n,\lambda}\r|\leq \frac{3}{2}T}= CT\l(1+\frac{1}{\lambda^2}\sum_{k=0}^{\infty} I_{n,k}\r) \leq
C\l(T+\frac{T}{\lambda^2}\e^{-n^2/4}\r)
\end{equation}
for every $\lambda,T>0$ and every $n\in\N_0$.

\vspace*{2mm}
\noindent
\emph{Second term on the right-hand of \eqref{eq:zerlegunglemmazehn}.}
This term is estimated using the Chebyshev inequality:
\begin{equation}\label{eq:secondtermontherhs}
\P{\l|\Theta_\lambda\r|\geq \frac{1}{2}T}\leq C\frac{\lambda^2}{T^2}.
\end{equation}

Taking~\eqref{eq:firsttermontherhs} and~\eqref{eq:secondtermontherhs} together,  we finally arrive at the estimate
\begin{equation*}
\P{\l|X_{q^n}(1)\r|\leq T}\leq C\left(T+ \frac{T}{\lambda^2}\e^{-n^2/4}+\frac{\lambda^2}{T^2}\right)
\end{equation*}
for all $\lambda,T>0$.
Choosing
\begin{equation*}
\lambda=T^{3/4}\EXP{-\frac{1}{16}n^2}
\end{equation*}
we optimize this bound, thus completing the proof of Lemma~\ref{lemma:abschaetzungdmeins}.
\end{proof}

\section{Auxiliary lemmas}
\subsection{A multivariate Lindeberg theorem}
In this section we state and prove a multivariate version of the Lindeberg central limit theorem for arrays with infinite rows. Since we were not able to find a direct reference, we decided to give a full proof of this standard result.
The next proposition is a univariate Lindeberg CLT in which the rows are allowed to be infinite.
\begin{proposition}\label{lemma:lindeberg_CLT}
Let $(\xi_{n,k})_{n,k\in\N_0}$ be a triangular array whose $n$-th row $\xi_{n,1}, \xi_{n,2},\ldots$ consists of independent, zero mean random variables. Suppose that
\begin{itemize}
\item[(a)]
$\lim_{n\to\infty} \sum_{k=0}^\infty \VAR\l[\xi_{n,k}\r]=\sigma^2\in (0,\infty)$.
\item [(b)]
For all $\eps>0$, we have
$\sum_{k=1}^\infty \EW{\xi_{n,k}^2\ind_{\{|\xi_{n,k}|>\eps\}}} \ton 0$.
\end{itemize}
Then,
\begin{equation*}
\sum_{k=0}^\infty \xi_{n,k} \todistr N(0,\sigma^2).
\end{equation*}
\end{proposition}
\begin{proof}
We may assume that $\sigma^2=1$, otherwise divide the variables by $\sigma$. By (a), for every $n\in\N_0$ there exists an integer $m_n\geq n$ such that
\begin{equation}\label{eq:proof_lindeberg_1}
\sum_{k=m_n+1}^\infty \VAR[\xi_{n,k}] \leq \frac{1}{n}.
\end{equation}
Let us split the sum we are interested in as follows:
\begin{equation*}
	\sum_{k=0}^\infty \xi_{n,k}=\sum_{k=0}^{m_n} \xi_{n,k}+\sum_{k=m_n+1}^\infty \xi_{n,k}=:S_n+R_n.
\end{equation*}
It follows from Condition~(a) and~\eqref{eq:proof_lindeberg_1} that $\lim_{n\to\infty} \sum_{k=0}^{m_n} \VAR[\xi_{n,k}] =1$ and therefore the classical Lindeberg CLT (for finite rows) guarantees that $S_n$  converges to the standard normal distribution.
Using Slutsky's lemma it suffices to prove that $R_n \to 0$ in probability.  Chebyshev's inequality yields
\begin{align*}
\P{\l|\sum_{k=m_n+1}^\infty \xi_{n,k}\r|>\eps}
&\leq
\eps^{-2}\sum_{k=m_n+1}^\infty\VAR\l[\xi_{n,k}\r]
\leq
\frac{1}{n\eps^2}
\ton 0
\end{align*}
for every $\eps>0$, thus completing the proof.
\end{proof}

The next proposition is a multivariate Lindeberg CLT for arrays with infinite rows.

\begin{proposition}\label{lemma:multivariate_lindeberg_CLT}
Let $(V_{n,k})_{n,k \in  \N_0}$ be a triangular array whose $n$-th row  consists of countably many independent $\R^d$-valued random vectors $V_{n,1}, V_{n,2},\ldots$ with zero mean.
Assume that
\begin{itemize}
\item[(a)]
$\lim_{n\to\infty} \sum_{k=0}^\infty  \COV\l[V_{n,k}\r]
=\Sigma
$
for some symmetric and positive semi-definite $d\times d$-matrix $\Sigma=(r_{ij})_{i,j=1}^d$.
\item[(b)]
For all $\eps>0$,
$$
L_{n}(\eps):=\max_{i=1,\dots, d}\sum_{k=0}^\infty \EW{V^2_{n,k}(i)\ind_{|V_{n,k}(i)|>\eps}} \ton 0,
$$
 where $V_{n,k}= (V_{n,k}(1), \ldots, V_{n,k}(d))$ is the coordinate representation of the vector $V_{n,k}$.
\end{itemize}
Then,	writing $N_d[0,\Sigma]$ for a $d$-variate normal distribution with mean $0$ and covariance matrix $\Sigma$, we have
\begin{equation*}
	S_n:=\sum_{k=0}^\infty V_{n,k} \todistr N_d[0,\Sigma].
	\end{equation*}
\end{proposition}
\begin{proof}
	Let $Z=(Z_1,\ldots,Z_d)$ be an $\R^d$-valued random vector with $Z\sim N_d[0,\Sigma]$.
	By the Cram\'er--Wold device it suffices to show that for every $(v_1,\ldots,v_d)\in \R^d$,
	\begin{equation}\label{eq:aquivalenteBed}
	\sum_{i=1}^d S_n(i) v_i \todistr  \sum_{i=1}^d Z_i v_i \sim N\l[0,\sum_{i,j=1}^d v_iv_j r_{ij}\r].
	\end{equation}
To prove~\eqref{eq:aquivalenteBed}, we apply the univariate Lindeberg CLT stated in Proposition~\ref{lemma:lindeberg_CLT} to the random variables
\begin{equation*}
\xi_{n,k}:=\sum_{i=1}^d V_{n,k}(i)v_i, \quad n,k\in\N_0.
\end{equation*}
Since $V_{n,k}$ is centered, we have $\E V_{n,k}(i)=0$ and hence, $\E \xi_{n,k}=0$. Furthermore, it follows from condition (a) that
	\begin{align*}
	\sum_{k=0}^\infty \VAR{\xi_{n,k}}
	= \sum_{k=0}^\infty\VAR\l[\sum_{i=1}^d V_{n,k}(i)v_i\r]
	= \sum_{k=0}^\infty \sum_{i,j=1}^d v_iv_j
     	\COV\l[V_{n,k}(i),V_{n,k}(j)\r]
	\ton \sum_{i,j=1}^d  v_iv_j r_{ij}.
	\end{align*}
	It remains to verify condition (b) of Proposition~\ref{lemma:lindeberg_CLT}. For $n,k\in \N_0$ write
	\begin{equation*}
	W_{n,k}:=\max_{j=1,\dots, d} V_{n,k}^2(j).
	\end{equation*}
	Let $\|v\|$ be the Euclidean norm of $v\in\R^d$. Applying the Cauchy--Schwarz inequality several times we obtain
	\begin{align*}
	\sum_{k=0}^\infty \EW{\xi_{n,k}^2\ind_{\{|\xi_{n,k}|>\eps\}}}
	&=\sum_{k=0}^\infty \EW{\l(\sum_{j=1}^d V_{n,k}(j)v_j\r)^2\ind_{\l\{\l(\sum_{j=1}^d
	V_{n,k}(j)v_j\r)^2>\eps^2\r\}}}\\
	&\leq\Vert v \Vert^2 \sum_{j=1}^d\sum_{k=0}^\infty \EW{ V^2_{n,k}(j)\ind_{\l\{\sum_{l=1}^d
	V^2_{n,k}(l)>\eps^2\Vert v\Vert^{-2}\r\}}} \\
	&\leq\Vert v \Vert^2d\sum_{k=0}^\infty \EW{W_{n,k}\ind_{\l\{W_{n,k}>\eps^2d^{-1}\Vert
	v\Vert^{-2}\r\}}} \\
	&\leq\Vert v \Vert^2d\sum_{k=0}^\infty
	\sum_{j=1}^d\EW{V_{n,k}^2(j)\ind_{\l\{V^2_{n,k}(j)>\eps^2d^{-1}\Vert v\Vert^{-2}\r\}}} \\
	&\leq\Vert v \Vert^2d^2L_n\l(\frac{\eps^2}{d\Vert v\Vert^2}\r) \ton 0.
	\end{align*}
The Lindeberg condition is thus verified and~\eqref{eq:aquivalenteBed} follows from Proposition~\ref{lemma:lindeberg_CLT}. Since $v\in\R^d$ was arbitrary, the proof is complete.
\end{proof}

\subsection{A uniform estimate for power series} In this section we prove the estimate~\eqref{eq:uniform_abel}.
\begin{lemma}\label{lem:uniform_abel}
Assume that $c_k\neq 0$ for all $k\in\N_0$ and, as always, that~\eqref{eq:c_k_reg_var} holds. There exist constants $C>0$ and $n_0\in \N$ such that
\begin{equation*}
\sum_{j=k}^{\infty} \binom{j}{k}\left|\frac{c_j}{c_k}\right|\l(1-q^{n+2}\r)^{j-k}
\leq \e^{C n(k+1)}
\end{equation*}
for every $k\in\N_0$ and $n\geq n_0$. 
\end{lemma}
\begin{proof}
For every fixed $k\in\N_0$, it is possible to deduce the claim of the lemma from Theorem~\ref{theo:abel}. What makes the following proof difficult, is the necessity to obtain an estimate which is uniform in $k$.  
For all $j,k\in \N_0$ with $j\geq k$ we have
\begin{equation}\label{eq:est_binomial}
\binom{j}{k}\leq \frac{j^k}{k!}.
\end{equation}
Fix some sufficiently small $\eps>0$ and introduce the abbreviations
\begin{equation}\label{key}
\gamma_1:=\frac{\gamma-1-\eps}{2} \quad \textrm{and} \quad
\gamma_2:=\frac{\gamma-1+\eps}{2}.
\end{equation}
The case $k=0$ follows directly from Theorem~\ref{theo:abel}, so let $k\geq 1$ in the following. Using \eqref{eq:est_binomial} and \eqref{eq:c_k_reg_var} we arrive at 
\begin{equation}\label{eq:first_estimate}
\sum_{j=k}^{\infty} \binom{j}{k}\left|\frac{c_j}{c_k}\right|\l(1-q^{n+2}\r)^{j-k}  \\
\leq
C\frac{\l(1-q^{n+2}\r)^{-k}}{k^{\gamma_1}k!}\sum_{j=k}^{\infty} j^{k+\gamma_2}\l(1-q^{n+2}\r)^{j}.
\end{equation}
Observe that the function $z\mapsto z^{k+\gamma_2}\l(1-q^{n+2}\r)^{z}$, $z\geq 0$, is unimodal (first increasing, then decreasing) and the point, where it attains its maximum, satisfies
\begin{equation}\label{eq:est_jnk}
j_{n,k}:=\arg\max_{z\geq 0} z^{k+\gamma_2}(1-q^{n+2})^z =-\frac{k+\gamma_2}{\log(1-q^{n+2})}\leq  (k+\gamma_2)q^{-(n+2)} 
\end{equation}
for all $k\in\N$, $n\geq n_0$, 
where we have used that $-\frac{1}{\log(1-x)}\leq \frac{1}{x}$ for all $0\leq x<1$.
Using the unimodality and estimating Riemann sums by Riemann integrals, we arrive at
\begin{equation}\label{eq:est_sumbyint}
\sum_{j=k}^\infty j^{k+\gamma_2}\l(1-q^{n+2}\r)^j \\
\leq 2j_{n,k}^{k+\gamma_2} \l(1-q^{n+2}\r)^{j_{n,k}}+\INT{0}{\infty}
x^{k+\gamma_2}\l(1-q^{n+2}\r)^x \D x.
\end{equation}

\vspace*{2mm}
\noindent
\textit{First term on the right-hand side of \eqref{eq:est_sumbyint}.}
Using \eqref{eq:est_jnk}  leads to
\begin{align}
C\frac{(1-q^{n+2})^{-k}}{k^{\gamma_1}k!} j_{n,k}^{k+\gamma_2} \l(1-q^{n+2}\r)^{j_{n,k}} 
&\leq C \frac{(k+\gamma_2)^{k+\gamma_2}(1-q^{n+2})^{j_{n,k}-k}}{k^{\gamma_1}k!}q^{-(n+2)(k+\gamma_2)} \label{eq:est_firsttermsum}\\
&\leq C \frac{(k+\gamma_2)^{k+\gamma_2}}{k^{\gamma_1}k!}q^{-(n+2)(k+\gamma_2)} \notag
\end{align}
since $(1-q^{n+2})^{j_{n,k}-k}\leq 1$ since $j_{n,k}\geq k$ for $n\geq n_0$ by~\eqref{eq:est_jnk}. By the Stirling formula, the right-hand side can
be estimated by $\e^{Cn(k+1)}$.
 
\vspace*{2mm}
\noindent
\textit{Second term on the right-hand side of \eqref{eq:est_sumbyint}.} Using the substitution $u=xq^{n+2}$ and the inequality $(1-x)^{1/x} \leq \e^{-1}$, $x\in (0,1)$, we obtain
\begin{align*}
\INT{0}{\infty}
x^{k+\gamma_2}\l(1-q^{n+2}\r)^x \D x 
& =
q^{-(n+2)(k+\gamma_2+1)}\INT{0}{\infty}
u^{k+\gamma_2}\l(1-q^{n+2}\r)^{uq^{-(n+2)}} \D u \\
& \leq q^{-(n+2)(k+\gamma_2+1)}\INT{0}{\infty}
u^{k+\gamma_2}\e^{-u} \D u \\
&= q^{-(n+2)(k+\gamma_2+1)} \Gamma(k+\gamma_2+1).
\end{align*}
Hence, 
\begin{align}\label{eq:est_secondtermsum}
\begin{split}
&C\frac{\l(1-q^{n+2}\r)^{-k}}{k^{\gamma_1}k!} \INT{0}{\infty} x^{k+\gamma_2}\left(1-q^{n+2}\right)^x \D x \\
&\quad \leq C (1-q^{n+2})^{-k} q^{-(n+2)(k+\gamma_2+1)}
\frac{\Gamma(k+\gamma_2+1)}{k^{\gamma_1}k!}\\
&\quad \leq C  2^k  q^{-(n+2)(k+\gamma_2+1)}
\frac{\Gamma(k+\gamma_2+1)}{k^{\gamma_1}k!}\\
&\quad \leq \e^{Cn(k+1)},
\end{split}
\end{align}
where in the last line we have used that $\Gamma(x+s) \sim x^s\Gamma(x)$ as $x\to +\infty$ for every fixed $s$.

\vspace*{2mm}
Combining~\eqref{eq:first_estimate}, \eqref{eq:est_sumbyint}, \eqref{eq:est_firsttermsum} and \eqref{eq:est_secondtermsum} yields the statement of the lemma.
\end{proof}

\subsection{Estimates for \texorpdfstring{$\tilde F_{n,k}$}{tilde F(n,k)}}
The main result of this section is Lemma~\ref{lemma:asymptotik_Fnk} which provides estimates for the quantities
$$
\tilde{F}_{n,k} = \sum_{j=k}^\infty a_{n,j}^2 = \sum_{j=k}^\infty \frac{(1-q^n)^{2j}c^2_j}{v(1-q^n)}.
$$

\begin{lemma}\label{lemma:uniformconvergenceexponentialfunction}
Let $(A_n)_{n\in\N}$ be any sequence of positive real numbers such that $A_n = o(n)$ as $n\to\infty$.
Then,
\begin{equation*}
\lim_{n\to\infty} \sup_{x\in [0,A_n]}\l|\l(1-\frac{1}{n}\r)^{nx}\e^x-1\r|=0.
\end{equation*}
\end{lemma}
\begin{proof}
By taking the logarithm, it suffices to show that
\begin{equation*}
\lim_{n\to\infty} \sup_{x\in [0,A_n]} \l|nx\log\l(1-\frac{1}{n}\r)+x\r|= 0
\end{equation*}
or, equivalently,
\begin{equation*}
\lim_{n\to\infty} A_n  \l|n\log\l(1-\frac{1}{n}\r)+1\r|= 0.
\end{equation*}
The claim follows from the expansion
\begin{equation*}\label{eq:logminuseinsdurchn_laurentseries}
n \log\l(1-\frac{1}{n}\r) + 1  = -\frac{1}{2n}+o\l(\frac{1}{n}\r), \quad \textrm{as } n\to\infty,
\end{equation*}
together with the assumption $A_n= o(n)$.
\end{proof}
\begin{lemma}\label{lemma:asympt_incompl_gamma}
Let $\alpha\in\R$. We have
$$
\INT{z}{\infty }x^{\alpha}\e^{-2x}\D x  \sim \frac 12 z^{\alpha} \e^{-2z}, \quad \text{as }  z\to +\infty.
$$
\end{lemma}
\begin{proof}
Take the quotient of both expressions and apply L'H\^{o}pital's rule.
\end{proof}
\begin{lemma}\label{lemma:riemsumm_asymptotik}
Let $\alpha\in\R, q\in(0,1)$ and
$(y_n)_{n\in\N}$ be a sequence of positive real numbers such that $y_n\to +\infty$ but
$y_n = o(q^{-n/2})$ as $n\to\infty$. Then,
\begin{equation}\label{eq:proofwithoutL}
q^n\sum_{\substack{x\in q^n\N\\ x\geq y_n}} x^{\alpha}(1-q^n)^{2xq^{-n}}\sim \INT{y_n}{\infty }x^{\alpha}\e^{-2x}\D x \quad \textrm{as }n\to\infty.
\end{equation}
\end{lemma}
\begin{proof}
Splitting up the difference, we obtain
\begin{multline}\label{eq:split_uctone}
\l|q^n \sum_{\substack{x\in q^n\N \\ x\geq y_n}}
x^{\alpha}(1-q^n)^{2xq^{-n}}-\INT{y_n}{\infty}x^{\alpha}\e^{-2x}\D x\r| \\
 \leq
\l|q^n \sum_{\substack{x\in q^n\N \\ x\geq y_n}}x^{\alpha}\e^{-2x}-
\INT{y_n}{\infty}x^{\alpha}\e^{-2x}\D x\r|
 + q^n \sum_{\substack{x\in q^n\N \\ x\geq y_n}}
x^{\alpha}\l|(1-q^n)^{2xq^{-n}} -\e^{-2x}\r|.
\end{multline}
By Lemma~\ref{lemma:asympt_incompl_gamma} it suffices to show that both terms on the right-hand side are
of order $o(y_n^{\alpha}\e^{-2y_n})$.

\vspace*{2mm}
\noindent
\emph{First term on the right-hand side of \eqref{eq:split_uctone}.}
Since $y_n\to +\infty$ as $n\to\infty$, the function  $x\mapsto x^{\alpha}\EXP{-2x}$ is monotone decreasing for $x\geq y_n-q^n$ and $n$ sufficiently large. Estimating Riemann integrals by Riemann sums, we get
\begin{align}\label{eq:est_riemannsummandintegral}
\lefteqn{\l|q^n \sum_{\substack{x\in q^n\N \\ x\geq y_n}}x^{\alpha}\e^{-2x}-
\INT{y_n}{\infty}x^{\alpha}\e^{-2x }\D x\r|} \\
&\leq2\INT{y_n}{y_n+q^n} x^{\alpha}\e^{-2x }\D x +q^n \sum_{\substack{x\in q^n\N \\ x\geq y_n}}x^{\alpha}\e^{-2x }
-\INT{y_n}{\infty}x^{\alpha}\e^{-2x }\D x \notag\\
&\leq 2q^n y_n^{\alpha} \e^{-2y_n } + q^{n \alpha} q^n \l(\sum_{k=\lceil y_nq^{-n}\rceil}^\infty k^{\alpha}\e^{-2kq^n }-\sum_{k=\lceil y_nq^{-n}\rceil+1}^\infty k^{\alpha}\e^{-2kq^n }\r) \notag\\
&\leq C q^n y_n^{\alpha} \e^{-2y_n }+ q^{n \alpha} q^n \lceil y_n q^{-n} \rceil^{\alpha}\e^{-2q^n\lceil y_n q^{-n}\rceil } \notag\\
&\leq Cq^n y_n^\alpha\e^{-2y_n },\notag
\end{align}
which is $o\l(y_n^\alpha\e^{-2y_n }\r)$, as desired.
Therefore, we have established that for every positive sequence $y_n\to+\infty$,
\begin{equation}\label{eq:erste_Asymptotik}
q^n \sum_{\substack{x\in q^n \N \\ x\geq y_n}} x^\alpha \e^{-2x } \sim
\INT{y_n}{\infty} x^\alpha \e^{-2x }\D x
\quad \textrm{as }n\to\infty.
\end{equation}

\vspace*{2mm}
\noindent
\emph{Second term on the right-hand side of \eqref{eq:split_uctone}.} Once again we split up the sum:
\begin{multline}\label{eq:splituctoneone}
q^n \sum_{\substack{x\in q^n\N \\ x\geq y_n}}
x^{\alpha}\l|(1-q^n)^{2xq^{-n}} -\e^{-2x }\r| \\
=
q^n\sum_{\substack{x\in q^n\N \\ y_n\leq x\leq q^{-n/2}}}
x^{\alpha}\e^{-2x }\l|(1-q^n)^{2xq^{-n}}\e^{2x }-1\r|
+
q^n\sum_{\substack{x\in q^n\N \\ x > q^{-n/2}}}
x^{\alpha}\l|(1-q^n)^{2xq^{-n}}-\e^{-2x }\r|.
\end{multline}
The first term on the right hand side of \eqref{eq:splituctoneone} can be estimated using \eqref{eq:erste_Asymptotik} and Lemma \ref{lemma:uniformconvergenceexponentialfunction} as follows:
\begin{align*}
\lefteqn{q^n\sum_{\substack{x\in q^n\N \\ y_n\leq x\leq q^{-n/2}}}
x^{\alpha}\e^{-2x }\l|(1-q^n)^{2xq^{-n}}\e^{2x }-1\r|}\\
&\leq
\sup_{x\in [0,q^{-n/2}]}\l|(1-q^n)^{2xq^{-n}}\e^{2x }-1\r|
\cdot \Bigg(q^n
\sum_{\substack{x\in q^n\N \\ y_n\leq x\leq q^{-n/2}}}
x^{\alpha}\e^{-2x } \Bigg)
\\
&\leq
o(1)
\cdot \INT{y_n}{\infty} x^{\alpha}\e^{-2x }\D x
=
o\l(y_n^\alpha\e^{-2 y_n}\r),
\end{align*}
where we used that according to Lemma \ref{lemma:uniformconvergenceexponentialfunction},
\begin{equation*}
\lim_{n\to\infty}\sup_{x\in [0,q^{-n/2}]}\l|(1-q^n)^{2xq^{-n}}\e^{2x }-1\r|
=0.
\end{equation*}
The second term on the right-hand side of \eqref{eq:splituctoneone} can be estimated in the following way: using the inequality $(1-q^n)^{2x q^{-n}} \leq \e^{-2x}$, we obtain
\begin{multline*}
q^n\sum_{\substack{x\in q^n\N \\ x > q^{-n/2}}}
x^{\alpha}\l|(1-q^n)^{2xq^{-n}}-\e^{-2x}\r|
\leq q^n\sum_{\substack{x\in q^n\N \\ x\geq q^{-n/2}}}
x^{\alpha}\l((1-q^n)^{2xq^{-n}}+\e^{-2x}\r) \\
\leq 2 q^n \sum_{\substack{x\in q^n\N \\ x\geq q^{-n/2}}}
x^{\alpha}\e^{-2x }
\leq
C\INT{q^{-n/2}}{\infty} x^\alpha \e^{-2x }\D x,
\end{multline*}
which is easily seen to be $o(y_n^{\alpha} \e^{-2 y_n})$ by Lemma~\ref{lemma:asympt_incompl_gamma} and the assumption $y_n = o(q^{-n/2})$.
Summarizing, we have established \eqref{eq:proofwithoutL}.
\end{proof}

\begin{lemma}\label{lemma:asymptotik_Fnk}
Fix $q\in (0,1)$.
\begin{itemize}
\item[(i)] Let $\delta_0>0$ be arbitrary. Then, there is $C>0$ such that
\begin{equation*}
\tilde{F}_{n,k}=\sum_{j=k}^\infty \frac{(1-q^n)^{2j}c^2_j}{v(1-q^n)} \leq
C \e^{-kq^n}
\end{equation*}
for all $k\geq\delta_0 nq^{-n}$ and all $n\in\N$.
\item[(ii)] Let $(y_n)_{n\in \N}$ be a sequence of positive real numbers such that $y_n\to +\infty$ but $y_n = o(q^{-n/2})$ as $n\to\infty$.
For every $\epsilon>0$ there exists a constant $c>0$  such that
\begin{equation*}
\tilde{F}_{n,\lfloor y_nq^{-n}\rfloor}=\sum_{j=\lfloor y_n q^{-n}\rfloor}^\infty \frac{(1-q^n)^{2j}c^2_j}{v(1-q^n )} \geq
c q^{2 n\eps} y_n^{\gamma-1-\eps}\e^{-2 y_n}
\end{equation*}
for all $n\in\N$.
\end{itemize}
\end{lemma}
\begin{proof}
We start with general considerations that will be needed both for (i) and (ii).
Using \eqref{eq:c_k_reg_var}, we obtain that for all $k\in\N_0$, $n\in\N$,
\begin{align}
\label{eq:upbound_Fnktilde}
\begin{split}
\tilde{F}_{n,k}&=\sum_{j=k}^\infty\frac{(1-q^n)^{2j}c^2_j}{v(1-q^n)} \\
&=\frac{1}{\Gamma(\gamma)}\frac{1}{v(1-q^n)}\sum_{j=k}^\infty j^{\gamma-1}L(j) (1-q^n)^{2j} \\
&=\frac{1}{\Gamma(\gamma)}\frac{1}{v(1-q^n)}\sum_{\substack{x\in q^n\N \\ x\geq kq^n}} \l(\frac{x}{q^n}\r)^{\gamma-1}L\l(\frac{x}{q^n}\r) (1-q^n)^{2x q^{-n}} \\
&=\frac{L(q^{-n})}{\Gamma(\gamma)}\frac{q^{-n\gamma}}{v(1-q^n)}\sum_{\substack{x\in q^n\N \\ x\geq kq^n}} q^n x^{\gamma-1}\frac{L(xq^{-n})}{L(q^{-n})} (1-q^n)^{2 x q^{-n}}.
 \end{split}
\end{align}
From Theorem \ref{theo:abel} it follows that for some $C>0$ and all sufficiently large $n\in\N$,
\begin{equation}\label{eq:abel_appli}
\frac{1}{\Gamma(\gamma)}\frac{q^{-n\gamma}L(q^{-n})}{v(1-q^n)} \in [1/C, C].
\end{equation}
Indeed, the sequence on the left-hand side converges to $2^\gamma /\Gamma(\gamma)$, as $n\to\infty$. Fix $\epsilon>0$.
Since $L$ is slowly varying, for all sufficiently large $t$ we have the estimate
\begin{equation}\label{eq:L_slowly_var}
t^{-\eps}\leq L(t)\leq  t^\eps.
\end{equation}

\vspace*{2mm}
\noindent
\emph{Proof of (i).} Since for all $\alpha>0$ and $x>0$,
$\l(1-\frac{x}{\alpha}\r)^\alpha \leq \e^{-x}$,
we have
\begin{multline*}
q^n\sum_{\substack{x\in q^n\N\\ x\geq kq^n}} \frac{L(xq^{-n})}{L(q^{-n})} x^{\gamma-1}(1-q^n)^{2xq^{-n}}
\leq  q^{n(1-2\eps)}\sum_{\substack{x\in q^n\N\\ x\geq kq^n}} x^{\gamma-1+\eps}(1-q^n)^{2xq^{-n}} \\
\leq  q^{n(1-2\eps)}\sum_{\substack{x\in q^n\N\\ x\geq kq^n}} x^{\gamma-1+\eps}\e^{-2x}
\leq C q^{-2n\eps}\INT{kq^n}{\infty} x^{\gamma-1+\eps}\e^{-2x }\D x,
\end{multline*}
	where we used \eqref{eq:erste_Asymptotik} in the last step.
	Using Lemma~\ref{lemma:asympt_incompl_gamma}, we arrive at
	\begin{equation}\label{eq:proof_upperboundi}
	Cq^{-2n\eps}\INT{kq^{n}}{\infty} x^{\gamma-1+\eps} \e^{-2x }\D x
\leq Cq^{-2n\eps} (kq^n)^{\gamma-1+\eps}\e^{-2kq^n } \leq C\e^{-kq^n }, \\
	\end{equation}
where we have chosen $\eps>0$ sufficiently small and used the inequality $kq^n\geq \delta_0 n$.

\vspace*{2mm}
\noindent
\emph{Proof of (ii).}
Let $(y_n)_{n\in\N}$ be a positive sequence such that $y_n\to +\infty$ and $y_n = o(q^{-n/2})$, as $n\to\infty$.
By~\eqref{eq:L_slowly_var} and Lemma \ref{lemma:riemsumm_asymptotik} we have
\begin{multline*}
q^n\sum_{\substack{x\in q^n\N\\ x\geq y_n}} \frac{L(xq^{-n})}{L(q^{-n})} x^{\gamma-1}(1-q^n)^{2xq^{-n}}
\geq  q^{n(1+2\eps)}\sum_{\substack{x\in q^n\N\\ x\geq y_n}} x^{\gamma-1-\eps}(1-q^n)^{2xq^{-n}}\\
\geq c q^{2n\eps}\INT{y_n}{\infty} x^{\gamma-1-\eps}\e^{-2x }\D x
\geq cq^{2n\eps} y_n^{\gamma-1-\eps}\e^{-2y_n },
\end{multline*}
where in the last step we employed Lemma~\ref{lemma:asympt_incompl_gamma}.
Setting $k=\lfloor y_nq^{-n} \rfloor$ in \eqref{eq:upbound_Fnktilde} and using \eqref{eq:abel_appli} together with the above estimate yields (ii).
\end{proof}

%
%
%
%
%
%
%

\subsection{Lemmas on the monotone rearrangement}\label{sec:rearrangement}
Recall that $(b_{n,k}^2)_{k\in\N_0}$ is the monotone non-increasing rearrangement of the sequence $(a_{n,k}^2)_{k\in\N_0}$ that was defined in~\eqref{eq:def_ank} as follows:
$$
a_{n,k}^2 = \frac{c_k^2 (1-q^n)^{2k}}{v(1-q^n)}.
$$
The properties of $b_{n,k}^2$ are difficult to access directly. In this section we collect several lemmas on $b_{n,k}^2$.   We start by providing an estimate for the maximal term $b_{n,0}^2$.

\begin{lemma}\label{lemma:ankmaximum}
There exists $n_0\in\N$ such that
\begin{equation*}
b_{n,0}^2 = \max_{k\in\N_0} a^2_{n,k}\leq q^{n/2 \cdot (\gamma \wedge 1)} \quad \textrm{for all }n\geq n_0.
\end{equation*}
\end{lemma}
\begin{proof}
Fix any $0 < \eps < \min (\gamma/2, 1/4)$. By the regular variation property~\eqref{eq:c_k_reg_var}, there is $C>0$ such that
\begin{equation}\label{eq:est_ck}
c_k^2\leq  C k^{\gamma-1+\eps}\quad \textrm{for all } k\in \N.
\end{equation}
The subsequent estimates hold for all sufficiently large $n\in\N$. From Theorem~\ref{theo:abel} we deduce that for some sufficiently small $c>0$,
\begin{equation}\label{eq:est_v}
v(1-q^n) \geq c q^{-n(\gamma-\eps)}.
\end{equation}

Suppose first that $\alpha:= \gamma-1+\eps > 0$.  The maximum of the function $x\mapsto x^{\alpha} (1-q^n)^{2x}$ over $x>0$ is attained at $x_0=-\frac{\alpha}{2\log (1-q^n)} \sim \frac 12 \alpha q^{-n}$ and hence is bounded from above by $C q^{-\alpha n}$.
Using the definition of $a_{n,k}$  and then~\eqref{eq:est_ck} and \eqref{eq:est_v}, we obtain that for all  $k\in\N$,
\begin{align*}
a_{n,k}^2 = \frac{c_k^2 (1-q^n)^{2k}}{v(1-q^n)}   \leq \frac{C k^{\gamma-1+\eps}(1-q^n)^{2k}}{c q^{-n(\gamma-\eps)}}
\leq C q^{-n(\gamma-1+\eps)}q^{n(\gamma-\eps)}=Cq^{n(1-2\eps)} \leq q^{n/2},
\end{align*}
where in the last step we used that $\eps <1/4$. If $\gamma-1+\eps \leq  0$, then~\eqref{eq:est_ck} and \eqref{eq:est_v} yield
\begin{align*}
a_{n,k}^2 = \frac{c_k^2 (1-q^n)^{2k}}{v(1-q^n)}   \leq \frac{C k^{\gamma-1+\eps}(1-q^n)^{2k}}{c q^{-n(\gamma-\eps)}}
\leq C q^{n(\gamma-\eps)} \leq q^{\gamma n/2},
\end{align*}
where in the last step we used that $\eps < \gamma/2$. Finally, for all $\gamma>0$ and $k=0$ we have
$$
a_{n,0}^2 = \frac{c_0^2}{v(1-q^n)}  \leq C q^{n(\gamma-\eps)} \leq q^{\gamma n/2},
$$
thus completing the proof.
\end{proof}
To understand the next lemma, it is useful to consider the simple example in which $c_j^2 = j^{\gamma-1}$ for all $j\in\N_0$, where $\gamma > 1$. It is easy to check that the sequence $(a_{n,j}^2)_{j\in\N_0}$ is unimodal and the maximum is attained at $j_{n,max} \sim (\gamma-1) q^{-n}$. It is natural to conjecture that all terms $a_{n,j}^2$ with $j$ close to the modus should be larger than all terms with $j$ larger than, say $n^{\delta} q^{-n}$, where $\delta>0$. The next lemma makes this precise under the general regular variation condition~\eqref{eq:c_k_reg_var}.
\begin{lemma}\label{lemma:est_bnk}
	Let $0\leq\alpha<\delta<1$ and $(a_{n,k})_{n\in\N,k\in\N_0}$ be defined in \eqref{eq:def_ank}.
	There exist a constant $\zeta>0$ and a number $n_0\in \N$ such that
	\begin{equation*}
	\min_{\lfloor \zeta q^{-n}\rfloor \leq j\leq \lfloor \zeta q^{-n} \rfloor + \lfloor n^\alpha q^{-n}\rfloor + k}
	a^2_{n,j}\geq
	\max_{j\geq \lfloor n^\delta q^{-n}\rfloor +k} a^2_{n,j} \quad \textrm{for }n\geq n_0, k\in \N_0.
	\end{equation*}
\end{lemma}
\begin{proof}
	By \eqref{eq:c_k_reg_var} we have
	\begin{equation*}
	a^2_{n,j}=\frac{L(j)j^{\gamma-1}(1-q^n)^{2j}}{\Gamma(\gamma)v(1-q^n)} \quad
	\textrm{for all } n \in\N, j\in\N_0.
	\end{equation*}
	Choose sufficiently large $\zeta>0$ such that $
	\l(i^{\gamma}(1-q^n)^{2i}\r)_{i\geq \lfloor \zeta q^{-n}\rfloor}$
	is a monotone decreasing sequence for every $n\in\N$. To see that such $\zeta$ exists, one takes a quotient of two subsequent terms of the sequence and shows, by Taylor expansion,  that it is $>1$ for sufficiently large $\zeta$.

For $k\in\N_0$ consider the sequences
\begin{align*}
u_{n,k}&:=\min_{\lfloor \zeta q^{-n} \rfloor \leq i\leq \lfloor \zeta q^{-n} \rfloor + \lfloor n^\alpha q^{-n}\rfloor + k}\frac{L(i)i^{\gamma-1}(1-q^n)^{2i}}{\Gamma(\gamma)v(1-q^n)},\\
o_{n,k}&:=\max_{j\geq \lfloor n^\delta q^{-n}\rfloor + k}\frac{L(j)j^{\gamma-1}(1-q^n)^{2j}}{\Gamma(\gamma)v(1-q^n)}.
\end{align*}
By the Potter bound~\cite[Theorem~1.5.6 on p.~25]{bingham_book}, for every $0<\eps<1$ and $C>1$ there exists a number $n_1\in\N_0$ such that
\begin{equation}\label{eq:usingpottersbound}
\frac{L(i)}{L(j)}\geq \frac{1}{C}\min \l\{\l(\frac{i}{j}\r)^{\eps},\l(\frac{j}{i}\r)^{\eps}\r\} \quad \textrm{for all } i,j\geq \zeta q^{- n_1}.
\end{equation}
Let $n$ be sufficiently large. Using~\eqref{eq:usingpottersbound} for $j>i$ and that the sequence $
	\l(i^{\gamma-1+\eps}(1-q^n)^{2i}\r)_{i\geq \lfloor \zeta q^{-n}\rfloor}$  is also decreasing, we obtain
\begin{align*}
\frac{u_{n,k}}{o_{n,k}}
&= \min_{\lfloor \zeta q^{-n} \rfloor \leq i \leq \lfloor \zeta q^{-n} \rfloor + \lfloor n^\alpha q^{-n}\rfloor + k}
\min_{j\geq \lfloor n^\delta q^{-n}\rfloor + k}
\frac{L(i)i^{\gamma-1}(1-q^n)^{2i}}{L(j)j^{\gamma-1}(1-q^n)^{2j}}\\
&\geq c
\min_{\lfloor \zeta q^{-n} \rfloor \leq i\leq \lfloor \zeta q^{-n} \rfloor + \lfloor n^\alpha q^{-n}\rfloor + k}
\min_{j\geq \lfloor n^\delta q^{-n}\rfloor + k}
\l(\frac{i}{j}\r)^{\gamma-1+\eps} \frac{(1-q^n)^{2i}}{(1-q^n)^{2j}} \\
&=c \l(\frac{\lfloor \zeta q^{-n}\rfloor + \lfloor n^\alpha q^{-n}\rfloor + k}{\lfloor n^\delta q^{-n}\rfloor + k}\r)^{\gamma-1+\eps}
(1-q^n)^{2 \lfloor \zeta q^{-n}\rfloor  + 2 \lfloor n^\alpha q^{-n}\rfloor   - 2 \lfloor n^\delta q^{-n}\rfloor}.
\end{align*}
Applying Lemma \ref{lemma:uniformconvergenceexponentialfunction} to the second factor and the estimate
$$
n^{\alpha-\delta} \leq \frac{\lfloor \zeta q^{-n}\rfloor + \lfloor n^\alpha q^{-n}\rfloor + k}{\lfloor n^\delta q^{-n}\rfloor + k} \leq 2
$$
to the first one, we obtain
\begin{align*}
\frac{u_{n,k}}{o_{n,k}} \geq c
n^{-C}
\e^{2(n^{\delta}-n^{\alpha}-\zeta)}
\geq c	\e^{n^{\delta}-n^{\alpha}} \longrightarrow \infty
\end{align*}
as $n\to\infty$ since $0\leq\alpha<\delta<1$ by our assumption. Observe that the above estimates are uniform in $k\in\N_0$. This implies the statement of Lemma~\ref{lemma:est_bnk}.
\end{proof}

Next we derive two corollaries of Lemma~\ref{lemma:est_bnk}.
\begin{lemma}\label{lemma:est_bnk_cor1}
Let $0<\delta <1$. Then, there exist $\zeta>0$ and $n_0\in\N$ such that
\begin{equation*}
b^2_{n,k-1+\lfloor n^{\delta}q^{-n}\rfloor}\leq a^2_{n, \lfloor \zeta q^{-n} \rfloor + k} \quad \textrm{for all } n\geq n_0, k\in\N_0.
\end{equation*}
\end{lemma}

\begin{proof}
Lemma~\ref{lemma:est_bnk} with $\alpha=0$ and $\delta/2$ instead of $\delta$ yields a value $n_0\in\N$ such that for all $n\geq n_0$ and $k\in\N_0$,
\begin{equation}\label{eq:qqq1}
a_{n, \lfloor\zeta q^{-n}\rfloor + k} \geq \min_{\lfloor \zeta q^{-n}\rfloor \leq j \leq  \lfloor \zeta q^{-n}\rfloor + \lfloor q^{-n}\rfloor   + k} a_{n,j}^2\geq \max_{j\geq \lfloor n^{\delta/2} q^{-n}\rfloor + k} a_{n,j}^2.
\end{equation}
Suppose, by contraposition, that there exists $n\geq n_0$ and $k\in\N_0$ such that
\begin{equation}\label{eq:falsche_annahme}
b^2_{n,k-1+\lfloor n^{\delta}q^{-n}\rfloor }>a^2_{n, \lfloor \zeta q^{-n}\rfloor + k}.
\end{equation}
Consider the set
\begin{equation*}
A_{n,k}:=\{j\in\N_0 \colon a_{n,j}^2 > a^2_{n,\lfloor \zeta q^{-n}\rfloor + k}\}.
\end{equation*}
Since the sequence $(b^2_{n,j})_{j\in\N_0}$ is the monotone non-increasing rearrangement of $(a_{n,j}^2)_{j\in\N_0}$, it follows from \eqref{eq:falsche_annahme}
that the cardinality of $A_{n,k}$ satisfies
\begin{equation*}
\# A_{n,k}\geq k+\lfloor n^{\delta}q^{-n}\rfloor.
\end{equation*}
On the other hand, \eqref{eq:qqq1} implies that no index $j\geq \lfloor n^{\delta/2} q^{-n}\rfloor + k$ can be contained in the set $A_{n,k}$. This is a contradiction.
\end{proof}
Recall that $F_{n,k}\leq \tilde F_{n,k}$ by definition. The next lemma provides a converse inequality.
\begin{lemma}\label{lemma:est_bnk_cor2}
For every $0< \eps < 1$  there exists $n_0\in\N$ such that
	\begin{equation*}
	F_{n,k}\geq \tilde{F}_{n,k+ \lfloor n^\eps q^{-n}\rfloor} \quad \textrm{for all } n\geq n_0, k\in\N_0.
	\end{equation*}
\end{lemma}
\begin{proof}
Using Lemma \ref{lemma:est_bnk} with $\alpha=0$ and $\delta = \eps$ we find $\zeta>0$ and $n_0\in\N$ such that
\begin{equation}\label{eq:est_bnk_appli_rep}
\min_{\lfloor \zeta q^{-n} \rfloor \leq j\leq \lfloor\zeta q^{-n}\rfloor + \lfloor q^{-n}\rfloor + k}  a^2_{n,j} \geq \max_{j\geq \lfloor n^{\eps}q^{-n}\rfloor + k} a^2_{n,j} \quad \textrm{for all }n\geq n_0, k\in\N_0.
\end{equation}
Recall that $F_{n,k}$ can be obtained from the sum  $\sum_{j=0}^\infty a_{n,j}^2$ by excluding the $k$ largest summands.
But \eqref{eq:est_bnk_appli_rep} means that the terms $a_{n,j}^2$ with $j\geq k+ \lfloor n^\eps q^{-n}\rfloor$ are not among the $k$ largest terms and hence are too small to be taken
out of the sum defining $F_{n,k}$. Therefore, all terms $a_{n,j}^2$ with $j\geq k+ \lfloor n^\eps q^{-n}\rfloor$ are included in $F_{n,k}$. Since the sum of these terms is nothing but $\tilde{F}_{n,k+ \lfloor n^\eps q^{-n}\rfloor}$, the proof is complete.
\end{proof}

\bibliographystyle{plainnat}
\bibliography{random_taylor}
\end{document}